\documentclass[11pt, notitlepage]{article}
\usepackage{amssymb,amsmath,comment}
\catcode`\@=11 \@addtoreset{equation}{section}
\def\thesection{\arabic{section}}

\def\theequation{\thesection.\arabic{equation}}
\catcode`\@=12
\usepackage{colortbl}
\usepackage{a4wide}
\usepackage{parskip}

\newcommand{\ds} {\displaystyle}
\newcommand{\e}{\varepsilon}
\newcommand{\pa} {\partial}
\newcommand{\al} {\alpha}
\newcommand{\ba} {\beta}
\newcommand{\de} {\delta}

\newcommand{\Om} {\Omega}
\newcommand{\ra} {\rightarrow}
\newcommand{\rp} {\rightharpoonup}

\newcommand{\De} {\Delta}
\newcommand{\la} {\lambda}

\newcommand{\noi} {\noindent}
\newcommand{\na} {\nabla}

\newcommand{\mb} {\mathbb}
\newcommand{\mc} {\mathcal}

\newcommand{\ld} {\langle}
\newcommand{\rd} {\rangle}

\usepackage[all]{xy}
\catcode`\@=11
\def\theequation{\@arabic{\c@section}.\@arabic{\c@equation}}
\catcode`\@=12

\def\QED{\hfill {$\square$}\goodbreak \medskip}

\newtheorem{Theorem}{Theorem}[section]
\newtheorem{Lemma}[Theorem]{Lemma}
\newtheorem{Proposition}[Theorem]{Proposition}
\newtheorem{Corollary}[Theorem]{Corollary}
\newtheorem{Remark}[Theorem]{Remark}
\newtheorem{Definition}[Theorem]{Definition}

\begin{document}

\title
{Coron problem for nonlocal equations involving Choquard nonlinearity }

\author{ Divya Goel$^{\,1}$\footnote{e-mail: {\tt divyagoel2511@gmail.com}}, \ Vicen\c tiu D. R\u adulescu$^{\,{2,3}}$\footnote{e-mail: {\tt vicentiu.radulescu@imar.ro}} \
	and \  K. Sreenadh$^{\,1}$\footnote{
		e-mail: {\tt sreenadh@maths.iitd.ac.in}} \\ \\ $^1\,${\small Department of Mathematics, Indian Institute of Technology Delhi,}\\
{\small	Hauz Khaz, New Delhi-110016, India }
\\ $^2\,${\small Faculty of Applied Mathematics, AGH University of Science and Technology,}\\ {\small al. Mickiewicza 30, 30-059 Krak\'ow, Poland}\\  $^3\,${\small  Institute of Mathematics ``Simion Stoilow" of the Romanian Academy,}\\ {\small P.O. Box 1-764, 014700 Bucharest, Romania} }

\date{}

\maketitle

\begin{abstract}
We consider the following Choquard equation
\[
-\De u = \left(\int_{\Om}\frac{|u(y)|^{2^*_{\mu}}}{|x-y|^{\mu}}dy\right)|u|^{2^*_{\mu}-2}u,   \; \text{in}\;
\Om,\quad
 u = 0 \;    \text{ on } \pa \Om ,
\]
where $\Om$ is a smooth  bounded domain in $\mathbb{R}^N$ ($N\geq 3$),  $2^*_{\mu}=(2N-\mu)/(N-2)$. This paper is concerned with the existence of a positive  high-energy solution of the above problem in an annular-type domain when the inner hole is sufficiently small.
\\
\noi \textbf{Key words:} Choquard nonlinearity, Coron problem, stationary nonlinear Schr\"odinger-Newton equation, Riesz potential, critical exponent.
\\
\noi \textbf{2010 Mathematics Subject Classification}: 35A15, 35J60, 35J20.
\end{abstract}

\section{Introduction}
In this paper, we study the existence of a positive solution of the Choquard equation. More precisely, we consider the problem
$$
(P) \qquad -\De u = \left(\int_{\Om}\frac{|u(y)|^{2^*_{\mu}}}{|x-y|^{\mu}}dy\right)|u|^{2^*_{\mu}-2}u   \; \text{in}\;
\Om,\quad
 u = 0 \;    \text{ on } \pa \Om,
$$
where $\Om$ is a smooth  bounded domain in $\mathbb{R}^N( N\geq 3)$,  $2^*_{\mu}=\frac{2N-\mu}{N-2}$,  $0< \mu <N$.

The work on elliptic equations involving critical Sobolev exponent over non-contractible domains was initiated by J.-M.~Coron in 1983. Indeed,  Coron \cite{coron}  proved the existence of a positive solution of the following critical elliptic problem
\begin{align*}
(Q) \qquad -\Delta u = u ^{\frac{N+2}{N-2}},  \; u>0\;  \text{in}\;
\Om,\;
u = 0 \; \quad   \text{on} \; \; \pa \Om ,
\end{align*}
\noi where $\Om$ is a smooth bounded domain in $\mathbb{R}^N$ and satisfies the following conditions:
there exist constants $0<R_1<R_2<\infty$ such that
\begin{align}\label{co15}
\{ x \in \mathbb{R}^N ;\; R_1<|x|<R_2 \} \subset\Om, \qquad
\{ x \in \mathbb{R}^N ;\; |x|<R_1 \} \nsubseteq \overline{\Om}.
\end{align}
Later on, A.~Bahri and J.-M.~Coron \cite{bahri}  proved that if there exists a positive integer $d$ such that $H_d(\Om,\mathbb{Z}_2) \not =0 $  (where $H_d(\Om,\mathbb{Z}_2)$ the homology of dimension $d$ of $\Om$ with $\mathbb{Z}_2$ coefficients), then problem $(Q) $ has a positive solution.

V.~Benci and G.~Cerami \cite{benci1}  considered the equation
\begin{equation}\label{co16}
-  \De u +\la u  = u^{p-1} , \;
 u >0 \text{ in } \Om,\;  u=0 \text{ on } \pa \Om,
\end{equation}
\noi where $ \Om \subset \mathbb{R}^N, N\geq 3$ is a smooth bounded domain and $2<p <2^*$, $\la \in \mathbb{R}_+ \cup \{0\}$. With the help of  Ljusternik-Schnirelmann theory, Benci and Cerami showed  that  there exists a function $\overline{\la}: (2,\; 2^*)\ra \mathbb{R}_+ \cup \{0\}$ such that for all $\la \geq \overline{\la}(p)$, problem \eqref{co16} has at least $\text{cat}\,(\Om)$ distinct solutions.
 We cite \cite{benci2, benci3, benci4, dancer, squassina1, musso,  rey, struwe2} and the references therein for the work on the existence of  solutions over a non-contractible domain.

We recall that the  Choquard equation \eqref{co20}  was first introduced in  the pioneering work of H.~Fr\"ohlich \cite{froh} and S.~Pekar \cite{pekar} for the  modeling of  quantum polaron:
  \begin{equation}\label{co20}
  -\De u +u = \left(\frac{1}{|x|}* |u|^2\right) u \text{ in } \mathbb{R}^3.
  \end{equation}
As pointed out by Fr\"ohlich \cite{froh} and Pekar, this model corresponds to the study of  free electrons
in an ionic lattice interact with phonons associated to deformations of the
lattice or with the polarisation that it creates on the medium (interaction
of an electron with its own hole).  In the
approximation to Hartree-Fock theory of one component plasma,
Choquard used  equation \eqref{co20} to
describe an electron trapped in its own hole,

The Choquard equation is also known as the
Schr\"odinger-Newton equation
in models coupling the Schr\"odinger equation of
quantum physics together with nonrelativistic
Newtonian gravity. The equation can also be derived
from the Einstein-Klein-Gordon and Einstein-Dirac system. Such a
model was proposed for boson stars and for the collapse of galaxy
fluctuations of scalar field dark matter. We refer for details to A.~Elgart and B.~Schlein \cite{elgart}, D. Giulini and A. Gro\ss ardt \cite{giulini}, K.R.W.~Jones \cite{jones}, and F.E.~Schunck and E.W.~Mielke \cite{schunk}.  R.~Penrose \cite{penrose, penrose1} proposed equation \eqref{co20}  as
a model of self-gravitating matter in which quantum state reduction was understood as
a gravitational phenomenon.

   As pointed out by E.H.~Lieb \cite{ehleib}, Ph.~Choquard used equation \eqref{co20}  to study steady states of the one component plasma approximation in the Hartree-Fock theory.  Classification of solutions of  \eqref{co20}  was first studied by L.~Ma and L.~Zhao \cite{ma}. For the broad survey of Choquard equations we  refer to V.~Moroz and J.~Van Schaftingen \cite{moroz}  and references therein.

    Recently, F.~Gao and M.~Yang \cite{yang} studied the Brezis-Nirenberg type result for the following problem
\begin{equation}\label{co17}
-\De u = \la u+ \left(\int_{\Om}\frac{|u(y)|^{2^*_{\mu}}}{|x-y|^{\mu}}dy\right)|u|^{2^*_{\mu}-2}u \text{ in } \Om, \quad
u=0 \text{ on } \pa \Om,
\end{equation}
\noi where $0<\la$, $0<\mu<N$, $2^*_{\mu}=\frac{2N-\mu}{N-2}$, $\Om$ is a smooth bounded domain in $\mathbb{R}^N$ and
 $2^*_{\mu}$ is critical exponent in the sense of Hardy--Littlewood--Sobolev inequality \eqref{co9}. Authors   proved the Pohozaev identity for the equation \eqref{co17} and  used variational methods and the minimizers of the best constant $S_{H,L}$ (defined in \eqref{co10}) to show the existence, non-existence of solution depending on the range of $\la$. We cite F.~Gao {\it et al.} \cite{yangjmaa, gaoyang} for the  Choquard equation with critical exponent in the sense of Hardy--Littlewood--Sobolev inequality. However,
the existence and multiplicity of solutions of nonlocal equations  over  non-contractible domains is still  an open question.  Therefore, it is essential to study the existence of a positive solution of elliptic  equations involving convolution-type nonlinearity in non-contractible domains.

 Inspiring by these results,  we study in the present article  the  Coron  problem  for the problem $(P)$. More precisely, we show the  existence of a high-energy positive solution in a non-contractible bounded domain particularly an annulus when the inner hole is sufficiently small. The functional associated  with $(P)$  is not  $C^2$  when  $\mu > \min\{4, N\}$ and this makes the problem $(P)$ more challenging.

 In order to achieve the  desired aim  we  first prove the non-existence result using the   Pohozaev identity  for Choquard equation on $\mathbb{R}^N_+$.
 We  also prove the global compactness lemma for Choquard equation in bounded domains. In case of $\mu=0$, such a lemma   has been proved by M.~Struwe \cite{struwe1} and later generalized to the $p$-Laplacian case by Mercuri and Willem \cite{mercuri}.  In case of   $0<\mu<N$,  the method of defining  L\'evy concentration function is not useful.  In the present article we gave the proof of  global compactness Lemma \ref{cothm3}  by introducing the notion of Morrey spaces.  Finally, by  using the concentration-compactness principle together with the deformation lemma,   we prove the existence of high-energy positive solution. To the best of our knowledge, there is no work on Coron's problem for  Choquard equation.

 We now state the  main result of this paper.
\begin{Theorem}\label{cothm1}
Assume that $\Om$ is a bounded domain in $\mathbb{R}^N$ satisfying the condition \eqref{co15}.
If $\ds \frac{R_2}{R_1}$ is sufficiently large   then problem $(P)$ admits a positive high-energy solution.
\end{Theorem}
Turning to the layout of the paper, in  Section 2 we assemble  notations and preliminary results.  In section 3, we give the classification of all non negative solutions of Choquard equation. In section 4, we analyze the Palais-Smale sequences. In section 5, we prove our main result  Theorem \ref{cothm1}.

\section{Preliminary results }
This section is devoted to the variational formulation, Pohozaev identity and non-existence result. The outset of the variational framework starts from the following  Hardy--Littlewood--Sobolev inequality. We refer to E.H.~Lieb and M.~Loss \cite{leib} for more details.
 \begin{Proposition}
Let $t,r>1$ and $0<\mu <N$ with $1/t+\mu/N+1/r=2$, $f\in L^t(\mathbb{R}^N)$ and $h\in L^r(\mathbb{R}^N)$. There exists a sharp constant $C(t,r,\mu,N)$ independent of $f,h$, such that
\begin{equation}\label{co9}
\int_{\mathbb{R}^N}\int_{\mathbb{R}^N}\frac{f(x)h(y)}{|x-y|^{\mu}}~dxdy \leq C(t,r,\mu,N) \|f\|_{L^t}\|h\|_{L^r}.
\end{equation}
If $t=r=2N/(2N-\mu)$, then
\begin{align*}
C(t,r,\mu,N)=C(N,\mu)= \pi^{\frac{\mu}{2}}\frac{\Gamma(\frac{N}{2}-\frac{\mu}{2})}{\Gamma(N-\frac{\mu}{2})}\left\lbrace \frac{\Gamma(\frac{N}{2})}{\Gamma(\frac{\mu}{2})}\right\rbrace^{-1+\frac{\mu}{N}}.
\end{align*}
Equality holds in  \eqref{co9} if and only if $f/h\equiv constant$ and
\begin{align*}
h(x)= A(\gamma^2+|x-a|^2)^{(2N-\mu)/2}
\end{align*}
 for some $A\in \mathbb{C}, 0\neq \gamma \in \mathbb{R}$ and $a \in \mathbb{R}^N$.\QED
\end{Proposition}

We consider  the following functional space
\begin{align*}
D^{1,2}(\mathbb{R}^N):=\{u\in L^{2^*}(\mathbb{R}^N)\; : \; \na u \in L^2(\mathbb{R}^N) \},
\end{align*}
endowed with the norm defined as
\begin{equation*}
\|u\|:= \left(\int_{\mathbb{R}^N}|\na u|^2 dx \right)^{\frac{1}{2}}.
\end{equation*}
The space $D^{1,2}_0(\Om)$ is defined as the  closure of $C_c^{\infty}(\Om)$ in $D^{1,2}(\mathbb{R}^N) $.
\begin{Definition}
	A function $u \in  D^{1,2}_0(\Om)$ is said to be a solution of $(P)$ if $u$ satisfies
	\begin{align*}
	\int_{\Om}\na u \na \phi ~dx = \int_{\Om} \int_{\Om} \frac{|u(x)|^{2^{*}_{\mu}}|u(y)|^{2^{*}_{\mu}-2}u(y)\phi(y)}{|x-y|^{\mu}}~dxdy \text{ for all } \phi \in D^{1,2}_0(\Om).
	\end{align*}
\end{Definition}

\textbf{Notation.} We define $u_+= \text{max}(u,0)$ and $u_{-}= \text{max}(-u,0)$ for all $u \in D^{1,2}(\mathbb{R}^N) $. Moreover, we set $\mathbb{R}^N_+:=\{ x \in \mathbb{R}^N\; |\;  x_N>0\}$ and we denote by $*$ the standard convolution operator.

\noi Consider functionals $I: D^{1,2}_0(\Om)\ra \mathbb{R}$ and $I_{\infty}: D^{1,2}(\mathbb{R}^N)\ra \mathbb{R}$ as
 \begin{align*}
 I(u) &= \frac12 \int_{\Om} |\nabla u|^2~dx - \frac{1}{2.2^*_{\mu}} \int_{\Om}\int_{\Om} \frac{|u_+(x)|^{2^{*}_{\mu}}|u_+(y)|^{2^{*}_{\mu}}}{|x-y|^{\mu}}~dxdy, \quad  u \in D^{1,2}_0(\Om) \\
 I_{\infty}(u) &= \frac12 \int_{\mathbb{R}^N} |\nabla u|^2~dx- \frac{1}{2.2^*_{\mu}}  \int_{\mathbb{R}^N}\int_{\mathbb{R}^N}\frac{|u_+(x)|^{2^{*}_{\mu}}|u_+(y)|^{2^{*}_{\mu}}}{|x-y|^{\mu}}~dxdy, \quad  u \in D^{1,2}(\mathbb{R}^N).
 \end{align*}
By the Hardy--Littlewood--Sobolev inequality,  we have
\begin{align*}
 \left(\ds \int_{\mathbb{R}^N}\int_{\mathbb{R}^N}\frac{|u(x)|^{2^{*}_{\mu}}|u(y)|^{2^{*}_{\mu}}}{|x-y|^{\mu}}~dxdy\right)^{\frac{1}{2^{*}_{\mu}}}\leq C(N,\mu)^{\frac{2N-\mu}{N-2}}\|u\|_{L^{2^*}}^2,
\end{align*}
 where $2^*= \frac{2N}{N-2}$. This implies that $I \in C^1(D_0^{1,2}(\Om),\mathbb{R})$ and $I_{\infty}\in C^1(D^{1,2}(\mathbb{R}^N),\mathbb{R})$. The best constant for the embedding $D^{1,2}(\mathbb{R}^N)$ into  $L^{2^*}(\mathbb{R}^N)$ is defined as
\begin{align}\label{co11}
S=\inf_{ u \in D^{1,2}(\mathbb{R}^N) \setminus \{0\}} \bigg\{ \ds \int_{\mathbb{R}^N}|\nabla u|^2dx :\int_{R^N}|u|^{2^*}~dx=1\bigg\}.
\end{align}
Consequently, we define
\begin{align}\label{co10}
S_{H,L}= \inf_{ u \in D^{1,2}(\mathbb{R}^N)\setminus \{0\}} \left\{  \int_{\mathbb{R}^N}|\nabla u|^2 dx:\;  \int_{\mathbb{R}^{N}} \int_{\mathbb{R}^N } \frac{|u(x)|^{2^{*}_{\mu}}|u(y)|^{2^{*}_{\mu}}}{|x-y|^{\mu}}~dxdy=1 \right\}.
\end{align}
It was established by G.~Talenti \cite{talenti} that the  best constant $S$ is achieved if and only if $u$ is of the form
\begin{align*}
\left(
\frac{t}{t^2+|x-(1-t)\sigma|^2} \right)^{\frac{N-2}{2}} \text{ for } \sigma \in \Sigma:=\{x\in \mathbb{R}^N: |x|=1\} \text{ and } t \in (0,1] .
\end{align*}

Properties of the best constant $S_{H,L}$ were established by F.~Gao and M.~Yang \cite{yang}. We recall the following property.
\begin{Lemma}\label{lem5}
The constant $S_{H,L}$ defined in \eqref{co10} is achieved if and only if
\begin{align*}
u=C\left(\frac{b}{b^2+|x-a|^2}\right)^{\frac{N-2}{2}}\,,
\end{align*}
where $C>0$ is a fixed constant, $a\in \mathbb{R}^N$ and $b\in (0,\infty)$ are parameters. Moreover,
\begin{align*}
S_{H,L}=\frac{S}{C(N,\mu)^{\frac{N-2}{2N -\mu}}}\,,
\end{align*}
\noi where $S$ is defined  as in \eqref{co11}.
\end{Lemma}
The following property was established in \cite{yang}.
\begin{Lemma}
	For $N\geq 3$ and $0<\mu<N$. Then
	\begin{align*}
	\|.\|_{NL}:= \left(\ds \int_{\mathbb{R}^N}\int_{\mathbb{R}^N}\frac{|.|^{2^{*}_{\mu}}|.|^{2^{*}_{\mu}}}{|x-y|^{\mu}}~dxdy\right)^{\frac{1}{2.2^{*}_{\mu}}}
	\end{align*}
	defines a norm on $L^{2^*}(\mathbb{R}^N)$.
\end{Lemma}

\begin{Remark}
If we define \begin{align*}
S_{A}= \inf_{ u \in D^{1,2}(\mathbb{R}^N)\setminus \{0\}} \left\{  \int_{\mathbb{R}^N}|\nabla u|^2 dx:\;  \int_{\mathbb{R}^{N}} \int_{\mathbb{R}^N } \frac{|u_+(x)|^{2^{*}_{\mu}}|u_+(y)|^{2^{*}_{\mu}}}{|x-y|^{\mu}}~dxdy=1 \right\}
\end{align*}
then $S_A=S_{H,L}$.
\end{Remark}

\begin{Proposition}\label{prop1}
Let $u \in D_0^{1,2}(\Om)$ be an arbitrary solution of the problem
\begin{equation}\label{co35}
-\De u =     \left(\int_{\Om}\frac{|u_+(y)|^{2^*_{\mu}}}{|x-y|^{\mu}}dy\right)    |u_{+}|^{2^*_{\mu}-1} \text{ in } \Om,  \; u=0 \text{ on } \pa \Om\,.
\end{equation}
Then $$I(u)\geq \frac{1}{2}\left(\frac{N-\mu+2}{2N-\mu}\right) S_{H,L}^{\frac{2N-\mu}{N-\mu+2}} =: \beta\,.$$ Moreover, the same conclusion holds for the solution $u \in D^{1,2}(\mathbb{R}^N)$ of
\begin{align} \label{co34}
-\De u =  \left(\int_{\mathbb{R}^N}\frac{|u_+(y)|^{2^*_{\mu}}}{|x-y|^{\mu}}dy\right)        |u_{+}|^{2^*_{\mu}-1} \text{ in } \mathbb{R}^N.
\end{align}
\end{Proposition}
\begin{proof}
If $u$ is a solution of \eqref{co35} then  testing \eqref{co35} with $u_+$ , $u_{-}$ yields
\begin{align*}
\int_{\Om}|\nabla u_+|^2dx=  \int_{\Om}\int_{\Om}\frac{|u_+(x)|^{2^{*}_{\mu}}|u_+(y)|^{2^{*}_{\mu}}}{|x-y|^{\mu}}~dxdy \text{ and } \int_{\Om}|\nabla u_{-}|^2dx= 0 \text{ a.e. on } \Om.
\end{align*}
It follows that $$(S_A)^{\frac{2^*_{\mu}}{2^*_{\mu}-1}}\leq \ds \int_{\Om}\int_{\Om}\frac{|u_+(x)|^{2^{*}_{\mu}}|u_+(y)|^{2^{*}_{\mu}}}{|x-y|^{\mu}}~dxdy=\int_{\Om}|\nabla u_+|^2dx=\int_{\Om}|\nabla u|^2dx. $$
It follows that
\begin{align*}
I(u)\geq \left(\frac{1}{2}-\frac{1}{2.2^*_{\mu}}\right)(S_A)^{\frac{2^*_{\mu}}{2^*_{\mu}-1}}= \frac{1}{2}\left(\frac{N-\mu+2}{2N-\mu}\right) S_{H,L}^{\frac{2N-\mu}{N-\mu+2}}
\end{align*}
The proof is now complete.\QED
\end{proof}
\begin{Lemma}\label{colem1}({\bf Pohozaev identity})
Let $N\geq 3$ and assume that $ u\in D^{1,2}_0(\mb R^N_+)$ solves
\begin{equation}\label{co12}
-\De u =      \left(\int_{\mathbb{R}^{N}_{+}}\frac{|u_+(y)|^{2^*_{\mu}}}{|x-y|^{\mu}}dy\right)        |u_{+}|^{2^*_{\mu}-1} \text{ in } \mathbb{R}^N_+.
\end{equation}
Then the following equality holds
\begin{equation*}
\frac{1}{2} \int_{\pa  \mathbb{R}^N_+ }(x-x_0)\cdot \nu |\na u|^2 dS + \frac{N-2}{2}\int_{\mathbb{R}^N_+}|\na u|^2 dx =\ds  \frac{2N-\mu}{2.2^*_{\mu}}\int_{\mathbb{R}^N_+}\int_{\mathbb{R}^N_+} \frac{|u_{+}(x)|^{2^*_{\mu}}|u_{+}(y)|^{2^*_{\mu}}}{|x-y|^{\mu}} dx dy
\end{equation*}
 where $\nu$ is the unit outward normal to $\pa \Om$ and $x_0=(0,0,\dots, 1)$.
\end{Lemma}
\begin{proof}
First observe that any solution of problem \eqref{co12} is non-negative. This implies \begin{align*}
\na u=\na u^+\quad \mbox{a.e. on}\ \mathbb{R}^N_+.
\end{align*}
Extending $u=0 $ in $\mathbb{R}^N\setminus \mathbb{R}^N_+$  we have  $u \in W^{2,2}_{\text{loc}}(\mathbb{R}^N)$ (see Lemma \ref{colem2}). Now fix  $\varphi \in C^1_c(\mathbb{R}^N)$ such that $\varphi=1 $ on $B_1$. Let the function $ \varphi_{\la} \in D^{1,2}(\mathbb{R}^N) $ defined  for $\la \in (0, \infty)$ and $x\in \mathbb{R}^N$ by
$\varphi_{\la}(x)=\varphi(\la x)$. Multiplying \eqref{co12} with $((x-x_0)\cdot\na u)\varphi_{\la}$ and integrating over $\mathbb{R}^N_+$, we obtain
\begin{equation}
\begin{aligned}\label{co7}
\int_{\mathbb{R}^N_+}(-\De u)&((x-x_0)\cdot\na u)\varphi_{\la}(x) dx = \int_{\mathbb{R}^N_+} \left(\int_{\mathbb{R}^{N}_{+}}\frac{|u_+(y)|^{2^*_{\mu}}}{|x-y|^{\mu}}dy\right) |u_{+}|^{2^*_{\mu}-1} ((x-x_0)\cdot\na u)\varphi_{\la}dx
 \\ =& \int_{\mathbb{R}^N_+}\na \left((x-x_0)\int_{\mathbb{R}^N_+}\left(\frac{|u_{+}(y)|^{2^*_{\mu}}}{|x-y|^{\mu}} dy \right) |u_{+}(x)|^{2^*_{\mu}-1}\varphi_{\la}(x)u(x)\right) dx \\
 &  - \int_{\mathbb{R}^N_+}u(x) \na \left((x-x_0)\int_{\mathbb{R}^N_+}\left(\frac{|u_{+}(y)|^{2^*_{\mu}}}{|x-y|^{\mu}} dy \right) |u_{+}(x)|^{2^*_{\mu}-1}\varphi_{\la}(x)\right) dx
\end{aligned}
\end{equation}
Using the  divergence theorem on the  right-hand side of \eqref{co7}, we obtain
\begin{equation}\label{co2}
\begin{aligned}
 \int_{\mathbb{R}^N_+}(-\De u)&((x-x_0)\cdot\na u)\varphi_{\la}(x) dx =  \int_{\mathbb{R}^N_+} \left(\int_{\mathbb{R}^{N}_{+}}\frac{|u_+(y)|^{2^*_{\mu}}}{|x-y|^{\mu}}dy\right) |u_{+}|^{2^*_{\mu}-1} ((x-x_0)\cdot\na u)\varphi_{\la}dx  \\
 =&  -  \int_{\mathbb{R}^N_+}u(x) \na \left((x-x_0)\int_{\mathbb{R}^N_+}\left(\frac{|u_{+}(y)|^{2^*_{\mu}}}{|x-y|^{\mu}} dy \right) |u_{+}(x)|^{2^*_{\mu}-1}\varphi_{\la}(x)\right) dx.
\end{aligned}
\end{equation}
\noi Now consider the  integral
\begin{equation}\label{co21}
\begin{aligned}
 \int_{\mathbb{R}^N_+}u(x)& \na \left((x-x_0)\int_{\mathbb{R}^N_+}\left(\frac{|u_{+}(y)|^{2^*_{\mu}}}{|x-y|^{\mu}} dy \right) |u_{+}(x)|^{2^*_{\mu}-1}\varphi_{\la}(x)\right) dx\\
  =& \int_{\mathbb{R}^N_+}N u(x)
  \left(\int_{\mathbb{R}^N_+}\frac{|u_{+}(y)|^{2^*_{\mu}}}{|x-y|^{\mu}} dy \right) |u_{+}(x)|^{2^*_{\mu}-1}\varphi_{\la}(x) dx\\&  \quad + \int_{\mathbb{R}^N_+}(2^*_{\mu}-1) u(x) \left(\int_{\mathbb{R}^N_+}\frac{|u_{+}(y)|^{2^*_{\mu}}}{|x-y|^{\mu}} dy \right) |u_{+}(x)|^{2^*_{\mu}-2}\varphi_{\la}(x)(\na u\cdot (x-x_0)) dx \\& \quad -\mu \int_{\mathbb{R}^N_+} u(x)\varphi_{\la}(x)\left(\int_{\mathbb{R}^N_+}\frac{|u_{+}(y)|^{2^*_{\mu}}(x-x_0)\cdot (x-y)}{|x-y|^{\mu +2}} dy\right)  |u_{+}(x)|^{2^*_{\mu}-1} dx\\&\quad + \la \int_{\mathbb{R}^N_+}\int_{\mathbb{R}^N_+} \frac{|u_{+}(y)|^{2^*_{\mu}}|u_{+}(x)|^{2^*_{\mu}}}{|x-y|^{\mu}}   (x-x_0)\cdot \na \varphi(\la x) ~dx dy .
\end{aligned}
\end{equation}
Taking into account  \eqref{co2} and \eqref{co21},   we have
\begin{equation}\label{co3}
\begin{aligned}
 2^*_{\mu} \int_{\mathbb{R}^N_+}&(x-x_0)\cdot\na u(x) \left( \int_{\mathbb{R}^N_+}\frac{|u_{+}(y)|^{2^*_{\mu}}}{|x-y|^{\mu}} dy \right) |u_{+}(x)|^{2^*_{\mu}-1}\varphi_{\la}(x)dx \\
  =& - N\int_{\mathbb{R}^N_+} u(x)\left(\int_{\mathbb{R}^N_+}\frac{|u_{+}(y)|^{2^*_{\mu}}}{|x-y|^{\mu}} dy \right) |u_{+}(x)|^{2^*_{\mu}-1}\varphi_{\la}(x) dx \\& \quad +\mu \int_{\mathbb{R}^N_+} u(x)\varphi_{\la}(x)\left(\int_{\mathbb{R}^N_+}\frac{|u_{+}(y)|^{2^*_{\mu}}(x-x_0).(x-y)}{|x-y|^{\mu +2}} dy\right)  |u_{+}(x)|^{2^*_{\mu}-1} dx\\&
  \quad - \la \int_{\mathbb{R}^N_+}\int_{\mathbb{R}^N_+} \frac{|u_{+}(y)|^{2^*_{\mu}}|u_{+}(x)|^{2^*_{\mu}}}{|x-y|^{\mu}}   (x-x_0)\cdot\na \varphi(\la x) ~dx dy .
\end{aligned}
\end{equation}
Now, interchanging the role of $x$ and $y$ in \eqref{co3} and combining the resultant equation with \eqref{co3}, we deduce that
\begin{equation}
\begin{aligned}
   \int_{\mathbb{R}^N_+}(x-x_0)\cdot\na u(x)& \int_{\mathbb{R}^N_+}\left(\frac{|u_{+}(y)|^{2^*_{\mu}}}{|x-y|^{\mu}} dy \right) |u_{+}(x)|^{2^*_{\mu}-1}\varphi_{\la}(x)dx \\=& \frac{\mu-2N}{2.2^*_{\mu}}\int_{\mathbb{R}^N_+}\int_{\mathbb{R}^N_+} \frac{|u_{+}(y)|^{2^*_{\mu}}|u_{+}(x)|^{2^*_{\mu}}}{|x-y|^{\mu}}  \varphi_{\la}(x) dx dy \\& - \frac{\la}{2^*_{\mu}} \int_{\mathbb{R}^N_+}\int_{\mathbb{R}^N_+} \frac{|u_{+}(y)|^{2^*_{\mu}}|u_{+}(x)|^{2^*_{\mu}}}{|x-y|^{\mu}}   (x-x_0)\cdot\na \varphi(\la x) ~ dx dy .
\end{aligned}
\end{equation}
Passing to the limit as $\la \ra 0 $ and using the dominated convergence theorem, we obtain that
\begin{equation}\label{co5}
\begin{aligned}
   \int_{\mathbb{R}^N_+}(x-x_0)\cdot\na u(x)& \left(\int_{\mathbb{R}^N_+}\frac{|u_{+}(y)|^{2^*_{\mu}}}{|x-y|^{\mu}} dy \right) |u_{+}(x)|^{2^*_{\mu}-1}dx \\&= \frac{\mu-2N}{2.2^*_{\mu}}\int_{\mathbb{R}^N_+}\int_{\mathbb{R}^N_+} \frac{|u_{+}(y)|^{2^*_{\mu}}|u_{+}(x)|^{2^*_{\mu}}}{|x-y|^{\mu}}  ~ dx dy .
\end{aligned}
\end{equation}
It is easily seen that
\begin{equation*}
\begin{aligned}
&\De u ((x-x_0)\cdot\na u)\varphi_{\la} \\ &=  \text{div} \left(\na u \varphi_{\la} \;(x-x_0)\cdot \na u \right)-\varphi_{\la}|\na u|^2-\varphi_{\la}(x-x_0)\cdot\na \left( \frac{|\na u|^2}{2}\right)-\la ((x-x_0)\cdot\na u)(\na \varphi(\la x)\cdot \na u)\\& =\ds \text{div} \left( \left(\na u (x-x_0)\cdot \na u  - (x-x_0) \frac{|\na u|^2}{2}\right) \varphi_{\la}\right)
 + \frac{N-2}{2} \varphi_{\la}|\na u|^2 \\&  \quad + \la \frac{|\na u|^2}{2}((x-x_0)\cdot\na \varphi(\la x)) -\la ((x-x_0)\cdot\na u)(\na \varphi(\la x)\cdot \na u).
\end{aligned}
\end{equation*}
Now, integrating by parts we obtain
\begin{equation*}
\begin{aligned}
\int_{\mathbb{R}^N_+} (\De u)((x-x_0)\cdot\na u)\varphi_{\la}\; dx& =  \int_{\pa \mathbb{R}^N_+}\left(\na u (x-x_0)\cdot \na u  - (x-x_0)  \frac{|\na u|^2}{2}\right) \varphi_{\la}\cdot\nu \; dS \\& \quad +\frac{N-2}{2} \int_{\mathbb{R}^N_+} \varphi_{\la}|\na u|^2 dx - \int_{\mathbb{R}^N_+}\la \frac{|\na u|^2}{2}((x-x_0)\cdot\na \varphi(\la x)) dx  \\& \quad -   \int_{\mathbb{R}^N_+} \la ((x-x_0)\cdot\na u)(\na \varphi(\la x)\cdot \na u) dx .
\end{aligned}
\end{equation*}
Noticing that  $\na u = (\na u\cdot \nu) \nu $ on $\pa \mathbb{R}^N_+$ and  employing   dominated convergence theorem  for $\la \ra 0$, we get that
\begin{equation}\label{co6}
\begin{aligned}
\int_{\mathbb{R}^N_+} (\De u)((x-x_0)\cdot\na u)= \frac{1}{2} \int_{\pa \mathbb{R}^N_+}|\na u|^2  (x-x_0)\cdot\nu \; dS  +\frac{N-2}{2} \int_{\mathbb{R}^N_+} |\na u|^2 dx  .
\end{aligned}
\end{equation}
From equation \eqref{co7}, \eqref{co5} and \eqref{co6} we have our desired result.
\QED
\end{proof}

We can now deduce the following Liouville-type theorem.
\begin{Theorem}\label{cothm2}
Let $N\geq 3, u\in D^{1,2}_0(\mb R^N_+)$  be any solution of
\begin{equation}\label{co1}
-\De u =   \left(\int_{\mathbb{R}^{N}_{+}}\frac{|u_+(y)|^{2^*_{\mu}}}{|x-y|^{\mu}}dy\right)     |u_{+}|^{2^*_{\mu}-1} \text{ in } \mathbb{R}^N_+.
\end{equation}
Then $u\equiv 0$ on $\mathbb{R}^N_+$.
\end{Theorem}
\begin{proof}
If $u$ is a solution of \eqref{co1} then
\begin{equation*}
\int_{\mathbb{R}^N_+}\na u\cdot \na \phi\;  dx - \int_{\mathbb{R}^N_+}\int_{\mathbb{R}^N_+} \frac{|u_{+}(x)|^{2^*_{\mu}}|u_{+}(y)|^{2^*_{\mu}-1}\phi(y)}{|x-y|^{\mu}}   \; dx\;  dy  \qquad \text{ for all } \phi \in D^{1,2}_0(\mb R^N_+).
\end{equation*}
Taking $\phi=u_{-}$ we obtain $u_{-}=0$ a.e. on $\mathbb{R}^N$. This implies that $u$ is a non-negative solution of \eqref{co1}. Now, by Lemma \ref{colem1} we have
\begin{equation*}
\int_{\pa \mathbb{R}^N_+}|\na u|^2  (x-x_0)\cdot\nu \; dS =0.
\end{equation*}
  But $(x-x_0)\cdot\nu >0$ for $x \in \pa \mathbb{R}^N_+ $. Since $u$ is a non-trivial solution, we get a contradiction from the Hopf boundary point lemma. Hence, $u\equiv0$ on $\mathbb{R}^N_+$.\QED
\end{proof}

\section{Classification of solutions  }
In this section we will discuss the regularity and  classification of  non-negative solutions of the following equation: 
\begin{align}\label{co32}
-\De u =\left(|x|^{\mu-N}*|u|^{p}\right) |u|^{p-2}u  \text{ in } \mathbb{R}^N,
\end{align}
where $p:= \frac{N+\mu}{N-2}$ and $0<\mu <N$.
Consider the following integral system of equations:
\begin{equation}\label{co39}
\begin{aligned}
& u(x)=\int_{\mathbb{R}^N}\frac{u^{p-1}(y)v(y)}{|x-y|^{N-2}}~dy,  u\geq 0 \text{ in } \mathbb{R}^N \\
& v(x)=\int_{\mathbb{R}^N}\frac{u^p(y)}{|x-y|^{N-\mu}}~dy,  v\geq 0 \text{ in } \mathbb{R}^N.
\end{aligned}
\end{equation}	
 We note that if $u\in D^{1,2}(\mathbb{R}^N)$, then $u,v$ defined above  is in $L^{\frac{2N}{N-2}}(\mathbb{R}^N) \times L^{\frac{2N}{N-\mu}}(\mathbb{R}^N)$. First we will discuss the  regularity of non-negative solutions of \eqref{co32}. In this regard, we will prove the following Lemma:
\begin{Lemma}\label{colem2}
	Let $u \in D^{1,2}(\mathbb{R}^N)$ be a non-negative solutions of \eqref{co32} then  $u \in W^{2,s}_{\text{loc}}(\mathbb{R}^N)$ for all $1 \leq s<\infty$.
\end{Lemma}
\begin{proof}
	Let $u\in D^{1,2}(\mathbb{R}^N)$ be a non-negative solution of \eqref{co32}  
	Now following the same approach  as in proof of  \cite[Lemma 3.1]{huang}, we have $ (u,v) \in L^r(\mathbb{R}^N)\times L^s(\mathbb{R}^N)$ for all $1<r, s<\infty$. In particular, $u^{p} \in L ^{\frac{N}{\mu}}(\mathbb{R}^N)$, and now using the boundedness of  Riesz potential  operator, we have  $|x|^{\mu-N}*u^{p} \in L^{\infty}(\mathbb{R}^N) $. Thus, from \eqref{co32}, we have
	\begin{align*}
	|-\Delta u| \leq C|u|^{p-1}.
	\end{align*}
	By classical elliptic regularity theory for subcritical problems in local bounded domains, we have $u \in W^{2,s}_{\text{loc}}(\mathbb{R}^N)$ for any $1\leq s < \infty$.  \QED
	\end{proof}

Next, we will discuss the classification of all positive solutions of the following system of integral equations:
\begin{equation}\label{co33}
\begin{aligned}
& u(x)=\int_{\mathbb{R}^N}\frac{u^{a}(y)v^b(y)}{|x-y|^{N-\al}}~dy , u>0 \text{ in } \mathbb{R}^N, \\
& v(x)=\int_{\mathbb{R}^N}\frac{u^c(y)v^d(y)}{|x-y|^{N-\beta}}~dy, v>0 \text{ in } \mathbb{R}^N,
\end{aligned}
\end{equation}
where $a\geq 0,\; b,c,d \in \{0\} \cup[1,\infty)$, $0<\al, \beta<N$.

Let $(u,v)\in L^{q_1}(\mathbb{R}^N)\times L^{q_2}(\mathbb{R}^N)$ be a solution of \eqref{co33}.  Now for all $\la \in \mathbb{R}$, we define $T_\la:= \{ (x_1, x_2,\cdots,x_n)\in \mathbb{R}^N : x_1 =\la  \}$  as the moving plane. Let $x^\la:= (2\la-x_1,x_2,\cdots,x_n), \; \Sigma_\la:= \{ (x_1, x_2,\cdots,x_n)\in \mathbb{R}^N : x_1 < \la   \}$  and $\Sigma_\la^\prime := \{ (x_1, x_2,\cdots,x_n)\in \mathbb{R}^N : x_1 \geq  \la   \}$ be the reflection of $\Sigma_\la$ about the plane $T_\la$. Moreover, define  $u_\la(y):= u(y^\la), \; v_\la(y)= v(y^\la)$. Immediately, we have the following property whose proof is just an elementary computation.
\begin{Lemma}\label{colem6}
	Assume that $(u,v)$ is a positive pair of solution of \eqref{co33}. Then
	\begin{equation*}
	\begin{aligned}
	& u(y^\la)-u(y)= \int_{\Sigma_\la} \left(  \frac{1}{|y-x|^{N-\al}}- \frac{1}{|y^\al-x|^{N-\al}}   \right)\left[u^a(x^\la)v^b(x^\la) - u^a(x)v^b(x)\right]~dx,\\
	& v(y^\la)-v(y)= \int_{\Sigma_\la} \left(  \frac{1}{|y-x|^{N-\ba}}- \frac{1}{|y^\al-x|^{N-\ba}}   \right)\left[u^c(x^\la)v^d(x^\la) - u^c(x)v^d(x)\right]~dx.
	\end{aligned}
	\end{equation*}
\end{Lemma}
\begin{Lemma}\label{colem7}
	There exists $\eta>0$ such that for all $\la<-\eta$,
	\begin{align*}
	u(y^\la)\geq u(y), \quad v(y^\la)\geq v(y), \text{ for all } y \in \Sigma_\la.
	\end{align*}
\end{Lemma}
\begin{proof}
	Define $\Sigma_\la^u:= \{  y \in \Sigma_\la: u(y)> u_\la(y)  \}, \;  \Sigma_\la^v:= \{  y \in \Sigma_\la: v(y)> v_\la(y)  \}$. By Lemma \ref{colem6}, we obtain
	\begin{align*}
	u(y^\la)-u(y)& = \int_{\Sigma_\la} \left(  \frac{1}{|y-x|^{N-\al}}- \frac{1}{|y^\la-x|^{N-\al}}   \right)\left[u^a(x^\la)v^b(x^\la) - u^a(x)v^b(x)\right]~dx\\
	& \leq  \int_{\Sigma_\la} \left(  \frac{1}{|y-x|^{N-\al}}- \frac{1}{|y^\la-x|^{N-\al}}   \right)\left[u_\la^a(v^b - v_\la^b)^+ +  v^b(u^a- u_\la^a)^+ \right]~dx\\
	& \leq  \int_{\Sigma_\la}  \frac{1}{|y-x|^{N-\al}} \left[u_\la^a(v^b - v_\la^b)^+ +  v^b(u^a- u_\la^a)^+ \right]~dx.
	\end{align*}	
	By the Hardy--Littlewood--Sobolev inequality, we obtain
	\begin{align*}
	\|u-u_\la\|_{L^{q_1}(\Sigma_\la^u)} \leq 	\|u-u_\la\|_{L^{q_1}(\Sigma_\la)} & \leq C \| u_\la^a(v^b - v_\la^b)^+ +  v^b(u^a- u_\la^a)^+ \|_{L^r(\Sigma_\la)}\\
	& \leq C \| u_\la^a(v^b - v_\la^b)\|_{L^r(\Sigma_\la^v)} +  \|v^b(u^a- u_\la^a) \|_{L^r(\Sigma_\la^u)},
	\end{align*}
	where $r= \frac{N{q_1}}{N+\al {q_1}}$. Now if $a,b>1$ then by H\"older's inequality, we get
	\begin{equation}\label{co36}
	\begin{aligned}
	\|u-u_\la\|_{L^{q_1}(\Sigma_\la^u)} & \leq  C \| u_\la^a v^{b-1}(v - v_\la)\|_{L^r(\Sigma_\la^v)} + C \|v^bu^{a-1}(u- u_\la) \|_{L^r(\Sigma_\la^u)}\\
	& \leq  C \| u_\la\|^a_{L^{q_1}(\Sigma_\la^v)} \|v^{b-1}(v - v_\la)\|_{L^s(\Sigma_\la^v)} + C \|v\|^b_{L^{q_2}(\Sigma_\la^u)} \| u^{a-1}(u- u_\la) \|_{L^t(\Sigma_\la^u)}\\
	& \leq  C \| u_\la\|^a_{L^{q_1}(\Sigma_\la^\prime)} \|v\|^{b-1}_{L^{q_2}(\Sigma_\la^v)}\|v - v_\la\|_{L^{q_2}(\Sigma_\la^v)} + C \|v\|^b_{L^{q_2}(\Sigma_\la)} \| u\|^{a-1}_{L^{q_1}(\Sigma_\la^u)}  \|u- u_\la \|_{L^{q_1}(\Sigma_\la^u)},
	\end{aligned}
	\end{equation}
	and if $0<a<1,\; b>1$ then we have
	\begin{equation}\label{co37}
	\begin{aligned}
	\|u-u_\la\|_{L^{q_1}(\Sigma_\la^u)} & \leq  C \| u_\la^a v^{b-1}(v - v_\la)\|_{L^r(\Sigma_\la^v)} + C \|v^b(u- u_\la)^a \|_{L^r(\Sigma_\la^u)}\\
	& \leq  C \| u_\la\|^a_{L^{q_1}(\Sigma_\la^v)} \|v^{b-1}(v - v_\la)\|_{L^s(\Sigma_\la^v)} + C \|v\|^b_{L^{q_2}(\Sigma_\la^u)} \| u- u_\la \|^a_{L^{q_1}(\Sigma_\la^u)}\\
	& \leq  C \| u_\la\|^a_{L^{q_1}(\Sigma_\la^\prime)} \|v\|^{b-1}_{L^{q_2}(\Sigma_\la^v)}\|v - v_\la\|_{L^{q_2}(\Sigma_\la^v)} + C \|v\|^b_{L^{q_2}(\Sigma_\la)}   \|u- u_\la \|_{L^{q_1}(\Sigma_\la^u)},
	\end{aligned}
	\end{equation}
	where $$s= \frac{rq_1}{q_1-ar},\; t =\frac{rq_2}{q_2-br}= \frac{q_1}{r}\; \mbox{and}\; \frac{b}{q_2}+\frac{a-1}{q_1}= \frac{\al}{N}.$$ Similarly, for  $c,d>1$ we have
	\begin{equation}\label{co38}
	\|v-v_\la\|_{L^{q_2}(\Sigma_\la^v)}  \leq  C \| v\|^d_{L^{q_2}(\Sigma_\la^\prime)} \|u\|^{c-1}_{L^{q_1}(\Sigma_\la^u)}\|u - u_\la\|_{L^{q_1}(\Sigma_\la^u)} + C \|u\|^c_{L^{q_1}(\Sigma_\la)}   \| v\|^{d-1}_{L^{q_2}(\Sigma_\la^v)}  \|v- v_\la \|_{L^{q_2}(\Sigma_\la^v)},
	\end{equation}
	where $q_1$ and $q_2$ are positive constant such that  $ \frac{d-1}{q_2}+\frac{c}{q_1}= \frac{\ba}{N}.$ Taking into account \eqref{co36}, \eqref{co37} and \eqref{co38}, for all $\la \in \mathbb{R}$ we have
	\begin{align*}
	\|u-u_\la\|_{L^{q_1}(\Sigma_\la^u)} & \leq  \bigg \{   \frac{C \| v\|^d_{L^{q_2}(\Sigma_\la^\prime)} \|u\|^{c-1}_{L^{q_1}(\Sigma_\la^u)} }{1-C \|u\|^c_{L^{q_1}(\Sigma_\la)}   \| v\|^{d-1}_{L^{q_2}(\Sigma_\la^v)} } \| u_\la\|^a_{L^{q_1}(\Sigma_\la^\prime)} \|v\|^{b-1}_{L^{q_2}(\Sigma_\la^v)} \\
	& \hspace{3cm} + C \|v\|^b_{L^{q_2}(\Sigma_\la)} \| u\|^{a-1}_{L^{q_1}(\Sigma_\la^u)}   \bigg \}\|u-u_\la\|_{L^{q_1}(\Sigma_\la^u)}.
	\end{align*}
	Using the fact that $(u,v) \in L^{q_1}(\mathbb{R}^N)\times L^{q_2}(\mathbb{R}^N)$, we can choose $\eta>0$ sufficiently large such that   for all $\la<-\eta$.
	\begin{align*}
	\frac{C \| v\|^d_{L^{q_2}(\Sigma_\la^\prime)} \|u\|^{c-1}_{L^{q_1}(\Sigma_\la^u)} }{1-C \|u\|^c_{L^{q_1}(\Sigma_\la)}   \| v\|^{d-1}_{L^{q_2}(\Sigma_\la^v)} } \| u_\la\|^a_{L^{q_1}(\Sigma_\la^\prime)} \|v\|^{b-1}_{L^{q_2}(\Sigma_\la^v)} + C \|v\|^b_{L^{q_2}(\Sigma_\la)} \| u\|^{a-1}_{L^{q_1}(\Sigma_\la^u)}\leq \frac{1}{2}.
	\end{align*}
	It follows that $\|u-u_\la\|_{L^{q_1}(\Sigma_\la^u)}=0$ and hence $\Sigma_\la^u$ must be measure zero and empty when $\la< -\eta$. In the similar manner, $\Sigma_\la^v$ must be of measure zero and empty when $\la< -\eta$. For all other cases, the proof follows analogously. This concludes the proof of  Lemma. \QED
\end{proof}
Now using the same assertions and arguments as in X.~Huang, D.~Li and L.~Wang \cite{huang} in combination  with Lemma \ref{colem7}, we have the following theorem.
\begin{Theorem}\label{cothm4}
	Assume that  $a\geq 0,\; b,c,d \in \{0\} \cup[1,\infty)$, $0<\al, \beta<N$ and  $(u,v)\in L^{q_1}(\mathbb{R}^N)\times L^{q_2}(\mathbb{R}^N)$ is a pair of positive solutions of  \eqref{co33}
	with $q_1$ and $q_2$ satisfies		
	\begin{align*}
	q_1,\; q_2>1, \qquad \frac{b}{q_2}+\frac{a-1}{q_1}= \frac{\al}{N}, \qquad  \frac{c}{q_1}+\frac{d-1}{q_2}= \frac{\beta}{N}.
	\end{align*}
	Then  $(u,v)$ is radially symmetric and monotone decreasing about some point in $\mathbb{R}^N$.
	Moreover,  if
	\begin{align*}
	b= \frac{1}{N-\beta}[(N+\al)-a(N-\al)], \quad c= \frac{1}{N-\al}[(N+\beta)-d(N-\beta)],
	\end{align*}
	then $(u,v)$ must be of  the form
	\begin{align*}
	u(x)= \left(\frac{d_1}{e_1+|x-x_1|^2}\right)^{\frac{N-\al}{2}}, \qquad v(x)= \left(\frac{d_2}{e_2+|x-x_2|^2}\right)^{\frac{N-\beta}{2}},
	\end{align*}
	for some constants $d_1,d_2, e_1,e_2>0$ and some $x_1,x_2 \in \mathbb{R}^N$.
\end{Theorem}
As an immediate corollary, we have the following result on radial symmetry of non-negative solutions of \eqref{co32}.
\begin{Corollary}\label{cothm5}
Every non-negative solution $u \in  D^{1,2}(\mathbb{R}^N)$ of equation \eqref{co32} is radially symmetric, monotone decreasing and of the form
\begin{align*}
u(x)= \left(\frac{c_1}{c_2+|x-x_0|^2}\right)^{\frac{N-2}{2}}.
\end{align*}
for some constants $c_1,c_2>0$ and some $x_0 \in \mathbb{R}^N$.
\end{Corollary}
\begin{proof}
	 Let $u$ be any non negative solution of the equation \eqref{co32}. Then by Lemma \ref{colem2}, we have $u \in W^{2,s}_{\text{loc}}(\mathbb{R}^N)$ for any $1\leq s < \infty$.
	 Hence, by   strong  maximum principle, we have $u$ is a positive function in $\mathbb{R}^N$. It implies 	$(u,v) \in L^{\frac{2N}{N-2}}(\mathbb{R}^N) \times L^{\frac{2N}{N-\mu}}(\mathbb{R}^N)$ is a positive  solution of the integral system \eqref{co39}.
	
Now employing Theorem \ref{cothm4} for $\al =2,\; a=p-1,\; b=1,\;\beta=\mu,\; c=p,\; d=0$ and using the fact $u \in D^{1,2}(\mathbb{R}^N)$, that is $u \in L^{\frac{2N}{N-2}}(\mathbb{R}^N)$ and $v\in L^{\frac{2N}{N-\mu}}(\mathbb{R}^N)$, we have the desired result.\QED
\end{proof}

\section{Palais-Smale analysis}
\begin{Lemma}\label{lem1}
Let $u_n \rp u $ be weakly convergent in $D^{1,2}(\mathbb{R}^N)$ and $u_n\ra u$ a.e. on $\mathbb{R}^N$. Then
\begin{align}\label{co31}
& (|x|^{-\mu}*|(u_n)_+|^{2^*_{\mu}})|(u_n)_+|^{2^*_{\mu}-2}(u_n)_+- (|x|^{-\mu}*|(u_n-u)_+|^{2^*_{\mu}})|(u_n-u)_+|^{2^*_{\mu}-2}(u_n-u)_+\nonumber\\& \quad \ra (|x|^{-\mu}*|u_+|^{2^*_{\mu}})|u_+|^{2^*_{\mu}-2}u_+ \text{ in } (D^{1,2}(\mathbb{R}^N))^{\prime}.
\end{align}
\end{Lemma}
\begin{proof}
Since $u_n \rp u $ weakly in $D^{1,2}(\mathbb{R}^N)$, there exists $M>0$ such that $\|u_n\| <M,\;  \text{ for all }  n \in \mathbb{N} $. Let $\phi \in D^{1,2}(\mathbb{R}^N)$ and
\begin{align*}
I= &\int_{\mathbb{R}^N}\left[ \left(|x|^{-\mu}*|(u_n)_+|^{2^*_{\mu}}\right)|(u_n)_+|^{2^*_{\mu}-2}(u_n)_+ \right. \\
 &\qquad -\left. \left(|x|^{-\mu}*|(u_n-u)_+|^{2^*_{\mu}}\right)|(u_n-u)_+|^{2^*_{\mu}-2}(u_n-u)_+\right] \phi~dx,
\end{align*}
then $I=I_1+I_2+I_3-2I_4$ where
\begin{align*}
 & I_1= \int_{\mathbb{R}^N} \left(|x|^{-\mu}* \left(|(u_n)_+|^{2^*_{\mu}} -|(u_n-u)_+|^{2^*_{\mu}}\right)\right)\\
 &\qquad \qquad \left(|(u_n)_+|^{2^*_{\mu}-2}(u_n)_+ -|(u_n-u)_+|^{2^*_{\mu}-2}(u_n-u)_+ \right)\phi ~dx,\\
 &
I_2= \int_{\mathbb{R}^N} \left(|x|^{-\mu}* |(u_n)_+|^{2^*_{\mu}} \right)|(u_n-u)_+|^{2^*_{\mu}-2}(u_n-u)_+ \phi ~dx, \\& I_3= \int_{\mathbb{R}^N} \left(|x|^{-\mu}* |(u_n-u)_+|^{2^*_{\mu}} \right)|(u_n)_+|^{2^*_{\mu}-2}(u_n)_+ \phi  ~dx,\\ & I_4= \int_{\mathbb{R}^N} \left(|x|^{-\mu}* |(u_n-u)_+|^{2^*_{\mu}} \right)|(u_n-u)_+|^{2^*_{\mu}-2}(u_n-u)_+ \phi ~dx.
\end{align*}
\noi  \textbf{Claim 1: } $\ds \lim_{n \ra \infty} I_1 = \int_{\mathbb{R}^N} \left( |x|^{-\mu}*|u_+|^{2^*_{\mu}} \right) |u_+|^{2^*_{\mu}-2}u_+ \phi~dx$.\\
Similar to the proof of the Brezis-Lieb lemma \cite{Leib}  we have,
\begin{align*}
|(u_n)_+|^{2^*_{\mu}}-|(u_n-u)_+|^{2^*_{\mu}} \ra |u_+|^{2^*_{\mu}} \text{ in } L^{\frac{2N}{2N-\mu}}(\mathbb{R}^N) \text { as } n\ra \infty.
\end{align*}
\noi Since the Hardy Littlewood-Sobolev inequality implies that the Riesz potential defines a linear continuous map from   $L^{\frac{2N}{2N-\mu}}(\mathbb{R}^N)$ to $L^{\frac{2N}{\mu}}(\mathbb{R}^N)$, we get
\begin{align}\label{co28}
|x|^{-\mu}* \left(|(u_n)_+|^{2^*_{\mu}} -|(u_n-u)_+|^{2^*_{\mu}}\right) \ra |x|^{-\mu}*|u_+|^{2^*_{\mu}} \text{ strongly  in } L^{\frac{2N}{\mu}} (\mathbb{R}^N) \text { as } n\ra \infty.
\end{align}
\noi Since both $|(u_n)_+|^{2^*_{\mu}-2}(u_n)_+ \phi \rp |u_+|^{2^*_{\mu}-2}u_+ \phi$  and  $|(u_n-u)_+|^{2^*_{\mu}-2}(u_n-u)_+ \phi \rp 0 $ converge weakly in $L^{\frac{2N}{2N-\mu}}(\mathbb{R}^N)$, we obtain
\begin{align}\label{co29}
|(u_n)_+|^{2^*_{\mu}-2}(u_n)_+ \phi -|(u_n-u)_+|^{2^*_{\mu}-2}(u_n-u)_+ \phi  \rp |u_+|^{2^*_{\mu}-2}u_+ \phi
\end{align}
weakly in $L^{\frac{2N}{2N-\mu}}(\mathbb{R}^N)$. Thus, Claim 1 follows from \eqref{co28} and \eqref{co29}. \\
\textbf{Claim 2:} $\ds \lim_{n \ra \infty} I_2 =  0 $.\\
Since $|(u_n)_+|^{2^*_{\mu}} \rp |(u)_+|^{2^*_{\mu}} \text{ weakly  in } L^{\frac{2N}{2N-\mu}}(\mathbb{R}^N)$, by the Hardy-Littlewood-Sobolev inequality \eqref{co9} we have
\begin{align}\label{co30}
|x|^{-\mu}*|(u_n)_+|^{2^*_{\mu}} \rp |x|^{-\mu}*|u_+|^{2^*_{\mu}} \text{ weakly in } L^{\frac{2N}{\mu}} (\mathbb{R}^N).
\end{align}
\noi We observe that $$|(u_n-u)_+|^{2^*_{\mu}-2}(u_n-u)_+ \phi \ra 0 \quad \mbox{a.e in}\ \mathbb{R}^N$$ and for any open subset $U\subset \mathbb{R}^N$, we have
\begin{align*}
\int_{U} \bigg||(u_n-u)_+|^{2^*_{\mu}-2}(u_n-u)_+ \phi \bigg|^{\frac{2N}{2N-\mu}}~ dx & \leq  \left(\int_{U}|(u_n-u)_+|^{2^*}~dx \right) ^{\frac{N-\mu +2}{2N-\mu}} \left(\int_{U}|\phi |^{2^*}~dx\right)^{\frac{N-2}{2N-\mu}}\\ & \leq \|u_n\|^{2^*(2^*_{\mu}-1)}\left(\int_{U}|\phi |^{2^*}~dx\right)^{\frac{N-2}{2N-\mu}} \\& \leq M \left(\int_{U}|\phi |^{2^*}~dx\right)^{\frac{N-2}{2N-\mu}}.
\end{align*}
\noi This implies that $\left\lbrace \bigg||(u_n-u)_+|^{2^*_{\mu}-2}(u_n-u)_+ \phi \bigg|^{\frac{2N}{2N-\mu}} \right\rbrace _n$ is equi-integrable in $L^1(\mathbb{R}^N)$. Hence, by the Vitali convergence theorem we get  that $|(u_n-u)_+|^{2^*_{\mu}-2}(u_n-u)_+ \phi \ra 0 $ strongly in $L^{\frac{2N}{2N-\mu}}(\mathbb{R}^N)$. This fact together with  \eqref{co30}  complete the proof of claim 2.\\
\textbf{Claim 3:}  $\ds \lim_{n \ra \infty} I_3 =  0 $.\\
Similar to the proof of claim 2 , we have
$|x|^{-\mu}*|(u_n-u)_+|^{2^*_{\mu}} \rp 0 \text{ weakly in } L^{\frac{2N}{\mu}} (\mathbb{R}^N)$ and\\
$|(u_n)_+|^{2^*_{\mu}-2}(u_n)_+ \phi \ra  |u_+|^{2^*_{\mu}-2} u_+ \phi  $ strongly in $  L^{\frac{2N}{2N-\mu}}(\mathbb{R}^N)$. Thus, claim 3 follows.\\
\textbf{Claim 4:}  $\ds \lim_{n \ra \infty} I_4 =  0 $.\\
Similar to the proof of claim 2 , we have
$ |x|^{-\mu}*|(u_n-u)_+|^{2^*_{\mu}} \rp 0 \text{ weakly in } L^{\frac{2N}{\mu}} (\mathbb{R}^N)$ and\\
$|(u_n-u)_+|^{2^*_{\mu}-2}(u_n-u)_+ \phi \ra  0 $  strongly in  $L^{\frac{2N}{2N-\mu}}(\mathbb{R}^N)$. Thus, claim 4 follows.
Hence $ I \ra \int_{\mathbb{R}^N} \left( |x|^{-\mu}*|u_+|^{2^*_{\mu}} \right) |u_+|^{2^*_{\mu}-2}u_+ \phi~dx $  that is,  \eqref{co31} holds.
\QED
\end{proof}

\begin{Lemma}\label{Lem2}
If $
u_n \rightharpoonup u \text{ weakly in }  D_0^{1,2}(\Om), \;  u_n \ra u  \text{ a.e  on } \Om,\;  I(u_n)\ra c,\;   I^{\prime}(u_n)\ra 0 \text{ in } (D_0^{1,2}(\Om))^{\prime}
$
then $I^{\prime}(u)=0$ and $v_n: =u_n-u$ satisfies
\begin{equation*}
 \|v_n\|^2=\|u_n\|^2-\|u\|^2+o(1), \;   I_{\infty}(v_n)\ra c- I(u), \;\text{and}\; I^{\prime}_{\infty}(v_n)\ra 0 \text{ in } (D_0^{1,2}(\Om))^{\prime}.
\end{equation*}
\end{Lemma}
\begin{proof}
{\bf Claim} : $I^{\prime}(u)=0$.\\
 As $u_n \rp u $   weakly in   $D_0^{1,2}(\Om)$ implies $|(u_n)_+|^{2^*_{\mu}} \rp |u_+|^{2^*_{\mu}}$ weakly in $L^{\frac{2N}{2N-\mu}}(\Om)$. Since Riesz potential is a linear continuous map from   $L^{\frac{2N}{2N-\mu}}(\Om)$ to $L^{\frac{2N}{\mu}}(\Om)$, we obtain that
\[
\int_{\Om}\frac{|(u_n)_+(y)|^{2^*_{\mu}}}{|x-y|^{\mu}}dy \rp \int_{\Om}\frac{|u_+(y)|^{2^*_{\mu}}}{|x-y|^{\mu}}dy  \; \text{ weakly in }\; L^{\frac{2N}{\mu}}(\Om)\]
Also, $|(u_n)_+|^{2^*_{\mu}-2}(u_n)_+ \rp |u_+|^{2^*_{\mu}-2}u_+ $ weakly in $L^{\frac{2N}{N-\mu+2}}(\Om)$.
Combining these facts we have
\[\left(\int_{\Om}\frac{|(u_n)_+(y)|^{2^*_{\mu}}}{|x-y|^{\mu}}dy\right) |(u_n)_+|^{2^{*}_{\mu}-2} (u_n)_+ \rp \left(  \int_{\Om}\frac{|u_+(y)|^{2^*_{\mu}}}{|x-y|^{\mu}}dy\right) |u_+|^{2^{*}_{\mu}-2} u_+ \; \text{ weakly in }\; L^{\frac{2N}{N+2}}(\Om).\]
This implies for any $\phi \in D_0^{1,2}(\Om)$,  we have
\begin{equation}\label{co8}
\begin{aligned}
\int_{\Om}\int_{\Om}& \frac{|(u_n)_+(x)|^{2^*_{\mu}}|(u_n)_+(y)|^{2^*_{\mu}-2}(u_n)_+(y) \phi(y) }{|x-y|^{\mu}}~dx dy
\\
& \ra \int_{\Om}\int_{\Om}\frac{|u_+(x)|^{2^*_{\mu}}|u_+(y)|^{2^*_{\mu}-2}u_+(y) \phi(y) }{|x-y|^{\mu}}~dx dy.
\end{aligned}
\end{equation}
Now, for $\phi \in D_0^{1,2}(\Om)$ consider
\begin{align*}
 \ld I^{\prime}(u_n)-I^{\prime}(u), \phi \rd &= \int_{\Om}\na u_n. \na \phi dx   -\int_{\Om}\int_{\Om}\frac{|(u_n)_+(x)|^{2^*_{\mu}}|(u_n)_+(y)|^{2^*_{\mu}-2}(u_n)_+ \phi(y) }{|x-y|^{\mu}} dx dy\\& - \int_{\Om}\na u. \na \phi dx + \int_{\Om}\int_{\Om}\frac{|u_+(x)|^{2^*_{\mu}}|u_+(y)|^{2^*_{\mu}-2}u_+ \phi(y) }{|x-y|^{\mu}} dx dy .
\end{align*}
Using \eqref{co8} and the fact that  $u_n \rp u $  weakly  in   $D_0^{1,2}(\Om)$ claim  follows. By the Brezis-Lieb lemma (see \cite{Leib}, \cite{yang}) we have
\begin{align*}
I_{\infty}(v_n) &= \frac{1}{2}\|u_n\|^2- \frac{1}{2}\|u\|^2 - \frac{1}{2.2^*_{\mu}}\int_{\Om}\int_{\Om}\frac{|(u_n-u)_+(x)|^{2^*_{\mu}}|(u_n-u)_+(y)|^{2^*_{\mu}} }{|x-y|^{\mu}} dx dy  +o(1)\\&=  \frac{1}{2}\|u_n\|^2 - \frac{1}{2.2^*_{\mu}}\int_{\Om}\int_{\Om}\frac{|(u_n)_+(x)|^{2^*_{\mu}}|(u_n)_+(y)|^{2^*_{\mu}} }{|x-y|^{\mu}} dx dy \\&\quad -  \frac{1}{2}\|u\|^2+ \frac{1}{2.2^*_{\mu}}\int_{\Om}\int_{\Om}\frac{|u_+(x)|^{2^*_{\mu}}|u_+(y)|^{2^*_{\mu}} }{|x-y|^{\mu}} dx dy  +o(1)\\& = I(u_n)-I(u)+o(1) \ra c- I(u).
\end{align*}
Now we will show that  $I_{\infty}^{\prime}(v_n) \ra 0 $ in
$(D_0^{1,2}(\Om))^{\prime}$. By Lemma \ref{lem1},  for any $ \phi \in D_0^{1,2}(\Om)$
\begin{align*}
 \ld  I_{\infty}^{\prime}(v_n),\phi \rd =  \ld I^{\prime}(v_n), \phi \rd= \ld I^{\prime}(u_n), \phi \rd -\ld I^{\prime}(u), \phi \rd +o(1) \ra 0.
\end{align*}
This implies $I_{\infty}^{\prime}(v_n) \ra 0 $ in
$(D_0^{1,2}(\Om))^{\prime}$. \QED
\end{proof}

\begin{Lemma}\label{lem2}
Let $\{y_n\}\subset \Om$ and $\{\la_n\}\subset (0,\infty)$ be such that
$\frac{1}{\la_n} dist (y_n,\pa \Om) \ra \infty. $
Assume the sequence $\{u_n\}$ and the rescaled sequence
\begin{align*}
f_n(x)= \la_n^{\frac{N-2}{2}}u_n(\la_nx+y_n)
\end{align*}
is such that
$f_n \rightharpoonup f \text{ weakly in }  D^{1,2}(\mathbb{R}^N),  f_n \ra f  \text{ a.e  on } \mathbb{R}^N , I_{\infty}(u_n)\ra c,  I^{\prime}_{\infty}(u_n)\ra 0 \text{ in } (D_0^{1,2}(\Om))^{\prime}$
then $ I_{\infty}^{\prime}(f)=0 $. Also, the sequence  $z_n(x)=u_n(x)-\la_n ^{\frac{2-N}{2}}f(\frac{x-y_n}{\la_n})$ satisfies
 $\|z_n\|^2=\|u_n\|^2-\|f\|^2+o(1),$  $I_{\infty}(z_n)\ra c- I_{\infty}(f)$ and $ I^{\prime}_{\infty}(z_n)\ra 0 \text{ in } (D_0^{1,2}(\Om))^{\prime}.$
\end{Lemma}
\begin{proof}
For $\phi\in C_c^{\infty}(\mathbb{R}^N)$ define $\phi_n(x):=\la_n ^{\frac{2-N}{2}}\phi(\frac{x-y_n}{\la_n})$. If $\phi\in C_c^{\infty}(B_k)$ then for large $n$, $\phi_n \in C_c^{\infty}(\Om)$. It implies
\begin{align*}
\ld I^{\prime}_{\infty}(f_n),\phi \rd = \ld I^{\prime}_{\infty}(u_n),\phi_n \rd \leq \|I^{\prime}_{\infty}(u_n)\| \|\phi_n\|=\|I^{\prime}_{\infty}(u_n)\| \|\phi\| \ra 0 .
\end{align*}
Hence, $I^{\prime}_{\infty}(f_n)\ra 0 $  as $n\ra \infty$ in $(D_0^{1,2}(B_k))^{\prime}$ for each $k$.\\
{\bf Claim :} $ I_{\infty}^{\prime}(f)=0 $ .\\
If $\phi\in C_c^{\infty}(\mathbb{R}^N)$ implies $\phi\in C_c^{\infty}(B_k)$ for some $k$. Now, using the fact $\frac{1}{\la_n} dist (y_n,\pa \Om) \ra \infty $ ,
$I^{\prime}_{\infty}(f_n)\ra 0 $ in $(D_0^{1,2}(B_k))^{\prime}$ and following the steps of Claim of Lemma \ref{Lem2}, we have $\ld I^{\prime}_{\infty}(f_n)-I_{\infty}^{\prime}(f) ,\phi \rd \ra 0 $ that is , claim holds. By the  Brezis-Lieb lemma (see \cite{Leib},  \cite{yang}),
\begin{align*}
I_{\infty}(z_n) &=  I_{\infty}(f_n-f) = I_{\infty}(u_n)-I_{\infty}(f)+o(1)  \ra c- I_{\infty}(f).
\end{align*}
\noi As $f_n \rightharpoonup f \text{ weakly in }  D^{1,2}(\mathbb{R}^N) $, we obtain
\begin{align*}
\text{ }\|z_n\|^2&= \int_{\mathbb{R}^N}|\na u_n(x)-\la_n ^{\frac{-N}{2}} \na f(\frac{x-y_n}{\la_n})|^2 dx
= \|u_n\|^2- \|f\|^2+o(1).
\end{align*}
By Lemma \ref{lem1} for any $\phi \in D_0^{1,2}(\Om)$, we have
\begin{align*}
\ld I^{\prime}_{\infty}(z_n),\phi \rd &= \left\langle I_{\infty}^{\prime}(u_n)-I_{\infty}^{\prime}\left (\la_n ^{\frac{2-N}{2}}f\left(\frac{.-y_n}{\la_n}\right)\right), \phi \right\rangle+o(1)\\&
= \left\langle I_{\infty}^{\prime}(u_n), \phi  \right\rangle+o(1) =o(1).
\end{align*}
This implies $I_{\infty}^{\prime}(z_n) \ra 0 $ in
$(D_0^{1,2}(\Om))^{\prime}$. \QED
\end{proof}
 Before proving the global compactness lemma for the Choquard equation, we will define the well-known Morrey spaces.
 \begin{Definition}
 	A measurable function $u: \mathbb{R}^N \ra \mathbb{R}$ belongs to Morrey  space $\mathcal{L}^{r,\gamma}(\mathbb{R}^N)$, with $r \in [1,\infty)$ and $\gamma \in [0,N]$, if and only if
 	\begin{align*}
 	\|u\|_{\mathcal{L}^{r,\gamma}(\mathbb{R}^N)}^r:= \sup_{R>0,\;   x \in \mathbb{R}^N} R^{\gamma- N}\int_{B(x,R)}|u|^r~ dy <\infty.
 	\end{align*}
 \end{Definition}
 Note that with the help of H\"older's inequality, we have  $L^{2^*}(\mathbb{R}^N) \hookrightarrow \mathcal{L}^{2,N-2}(\mathbb{R}^N)$.
\begin{Lemma}\label{cothm3}(\textbf{Global compactness lemma})
Let  $\{u_n\}_{n\in \mathbb{N}} \subset D_0^{1,2}(\Om)$ be such that $I(u_n)\ra c, I^{\prime}(u_n)\ra 0 $.
Then passing if necessary to a subsequence, there exists a solution $v_0 \in D_0^{1,2}(\Om) $ of
\begin{align}\label{co25}
-\Delta u= \left(\int_{\Om}\frac{|u_+(y)|^{2^*_{\mu}}}{|x-y|^{\mu}}dy\right)
|u_{+}|^{2^*_{\mu}-1} \text{ in } \Om
\end{align}
and (possibly) $k\in \mathbb{N}\cup \{0\}$, non-trivial solutions $\{v_1,v_2,...,v_k\}$ of
\begin{align}\label{co26}
-\Delta u= (|x|^{-\mu}*|u_{+}|^{2^*_{\mu}})|u_{+}|^{2^*_{\mu}-1} \text{ in } \mathbb{R}^N
\end{align}
 with $v_i \in D^{1,2}(\mathbb{R}^N) $ and $k$ sequences $\{y_n^i\}_{n\in \mathbb{N}} \subset \mathbb{R}^N$ and $\{\la_n^i\}_{n\in \mathbb{N}} \subset \mathbb{R}_+$  $i=1,2,\cdots k$, satisfying
 \begin{equation*}
 \begin{aligned}
 \frac{1}{{\la}_n^i} dist (y_n^i, \partial \Om) \ra \infty ,\; \text{and} \; \|u_n-v_0-\sum_{i=1}^{k}(\la_n^i)^{\frac{2-N}{2}}v_i((.-y_n^i)/\la_n^i)\| \ra 0 ,\; n \ra \infty,
 \end{aligned}
 \end{equation*}
 \begin{equation}\label{co13}
 \begin{aligned}
\|u_n\|^2\ra \sum_{i=0}^{k}\|v_i\|^2, n\ra \infty,\;\;\;
I(v_0)+\sum_{i=1}^{k} I_{\infty}(v_i)=c.
 \end{aligned}
 \end{equation}
\end{Lemma}
\begin{proof}
We divide the proof into several steps:\\
\textbf{Step 1:} By coercivity of the functional $I$, we get  $\{u_n\}$ is a bounded sequence in $D_0^{1,2}(\Om)$. It implies that there exists a $v_0 \in D_0^{1,2}(\Om)$ such that  $u_n \rp v_0$ weakly  in $D_0^{1,2}(\Om)$, $u_n\ra v_0$ a.e on $\Om$. By Lemma \ref{Lem2}, $I^{\prime}(v_0)=0$ and  $u_n^1=u_n-v_0$ such that
\begin{equation}\label{co22}
 \|u_n^1\|^2=\|u_n\|^2-\|v_0\|^2+o(1), \; I_{\infty}(u_n^1)\ra c- I(v_0) \; \text{and}\;  I^{\prime}_{\infty}(u_n^1)\ra 0 \text{ in } (D_0^{1,2}(\Om))^{\prime}.
\end{equation}
Moreover, there exists a constant $M_1>0$ such that $\|u_n^1\|< M_1$ for all $ n \in \mathbb{N}$.\\
\textbf{Step 2:}
If $\ds \int_{\Om}\int_{\Om}\frac{|(u_n^1)_+(x)|^{2^*_{\mu}}|(u_n^1)_+(y)|^{2^*_{\mu}} }{|x-y|^{\mu}} dx dy \ra 0$,
 then using the fact  that  $I^{\prime}(u_n)\ra 0$, it follows that $u_n^1\ra 0$ in $D_0^{1,2}(\Om)$ and we are done. \\
If $\ds \int_{\Om}\int_{\Om}\frac{|(u_n^1)_+(x)|^{2^*_{\mu}}|(u_n^1)_+(y)|^{2^*_{\mu}} }{|x-y|^{\mu}} dx dy \nrightarrow 0$ then we may assume that
\begin{align*}
\ds \int_{\Om}\int_{\Om}\frac{|(u_n^1)_+(x)|^{2^*_{\mu}}|(u_n^1)_+(y)|^{2^*_{\mu}} }{|x-y|^{\mu}} dx dy > \de, \quad \text{ for some } \de>0.
\end{align*}
 This on using Hardy--Littlewood-Sobolev inequality gives  $\|u_n^1\|_{L^{2^*}}> \de_1$ for all $n$ and for an appropriate   positive constant $\de_1$. Taking into account  that $u_n^1$ is a bounded sequence in $L^{2^*}(\mathbb{R}^N),\; L^{2^*}(\mathbb{R}^N) \hookrightarrow \mathcal{L}^{2,N-2}(\mathbb{R}^N)$, and Theorem 2 of G.~Palatucci and A.~Pisante \cite{palatucci},  we obtain
 \begin{align*}
c_2< \|u_n^1\|_{\mathcal{L}^{2,N-2}(\mathbb{R}^N)}<c_1, \text{ for  all } n.
 \end{align*}
 Thus, there exists a positive constant $C_0$ such that for all $n$ , we have
 \begin{align}\label{co18}
C_0< \|u_n^1\|_{\mathcal{L}^{2,N-2}(\mathbb{R}^N)}< C_0^{-1}.
 \end{align}
Now employing the definition of Morrey spaces  and \eqref{co18}, for each $n \in \mathbb{N}$ there exists $\{y_n^1,\; \la_n^1\} \in \mathbb{R}^N \times \mathbb{R}^+$ such that
\begin{align*}
0< \widehat{C_0}< \|u_n^1\|_{\mathcal{L}^{r,\gamma}(\mathbb{R}^N)}^2 - \frac{C_0^2}{2n}< (\la_n^1)^{-2}\int_{B(y_n^1,\la_n^1)}|u_n^1|^2~dy,
\end{align*}
for some suitable positive constant $\widehat{C_0}$. Now, define $f_n^1(x):= (\la_n^1)^{\frac{N-2}{2}}u_n^1(\la_n^1x+y_n^1)$.
Since $\|f_n^1\|=\|u_n^1\|$ thus $\|f_n^1\|<M_1$ for all $n\in \mathbb{N}$ and  we can assume that $f_n^1\rp v_1$ weakly  in  $D^{1,2}(\mathbb{R}^N), f_n^1\ra v_1$ a.e on $\mathbb{R}^N$. Moreover,
\begin{align*}
\int_{B(0,1)}|f_n^1|^2~ dx=  (\la_n^1)^{N-2}\int_{B(0,1)}|u_n^1(\la_n^1x+ y_n^1)|^2~dx= (\la_n^1)^{-2}\int_{B(y_n^1,\la_n^1)}|u_n^1(y)|^2~dy> \widehat{C_0}>0.
\end{align*}
Since, $D^{1,2}(\mathbb{R}^N) \hookrightarrow L^2_{\text{loc}}(\mathbb{R}^N)$  is compact, we have $\int_{B(0,1)}|v_1|^2~ dx> \widehat{C_0}>0$. It implies that $v_1 \not = 0$.\\
\textbf{Step 3:}  We claim that  $\la_n \ra 0$  and $y_n^1\ra  y_0 \in \overline{\Om}$. \\
Let if possible $\la_n \ra \infty $. As $\{u_n^1 \}$ is a bounded sequence in $D_0^{1,2}(\Om)$, it implies   $\{u_n^1\}$ is a bounded sequence in $L^2(\Om)$. Thus, if we define $\Om_n= \ds \frac{\Om- y_n^1}{\la_n^1}$ then
\begin{align*}
\int_{\Om_n} |f_n^1|^2~dx = \frac{1}{(\la_n^1)^2}\int_{\Om} |u_n^1|^2~dx \leq \frac{C}{\la_n^2}\ra 0.
\end{align*}
Contrary to this, using Fatou's lemma, we have $$0= \ds \liminf_{n\ra \infty}\int_{\Om_n}|f_n^1|^2~dx \geq \int_{\Om_n} |v_1|^2~dx.$$ This means that $v\equiv0$, which is not possible by step 2. Hence $\{\la_n^1\}$ is bounded in $\mathbb{R}$, that is, there exists $0\leq \la_0^1 \in \mathbb{R}$ such that $\la_n^1\ra \la_0^1$ as $n\ra \infty$. If $|y_n^1| \ra  \infty$ then for  any $x \in \Om$	and  large $n$, $\la_nx+y_n \not \in \overline{\Om}$. Since $u_n\in D_0^{1,2}(\Om)$ then $u_n^1(\la_nx+y_n) =0$ for all $x\in \Om$, it yields a contradiction to the assumption
$\|u_n\|_{NL}^{2.2^*_{\mu}} >\de>0 $. Therefore, $y_n^1$ is bounded, it implies that $y_n^1\ra y_0^1 \in \mathbb{R}^N$. Now let if  possible then $\la_n^1 \ra \la_0^1>0$ then $\Om_n\ra  \ds \frac{\Om- y_0^1}{\la_0^1} = \Om_0 \not = \mathbb{R}^N$. Hence  using the fact that $u_n^1 \rp 0$ weakly  in $D_0^{1,2}(\Om)$ we have  $f_n^1 \rp 0$ weakly in $D^{1,2}(\mathbb{R}^N)$ which is
not possible since   by step 2, $v_1 \neq 0$.
This implies $\la_n^1\ra 0 $. Arguing by contradiction, we assume that
\begin{align}\label{co19}
 y_0^1 \not \in \overline{\Om} .
\end{align}
In view of the fact that  $\la_n^1x+y_n^1 \ra y_0^1$ for all $x \in \Om$  as $n \ra \infty$. Now using \eqref{co19} we have $\la_n^1x+y_n^1  \not \in \overline{\Om}$ for all $ x \in \Om $ and $n$ large enough. It implies that $u_n^1(\la_n^1x +y_n^1)=0 $ for $n$ large enough, which is not possible. Therefore, $ y_0^1\in \overline{\Om}$. This completes the proof of claim and step 3. \\
\textbf{Step 4: } Assume that
\begin{align*}
\lim_{n\ra \infty} \frac{1}{\la_n^1}{\rm dist}\,(y_n^1,\pa \Om) \ra \al < \infty.
\end{align*}
Then $v_1$ is a solution of \eqref{co1} and
by Theorem \ref{cothm2} we have $v_1 \equiv 0$, which is not possible. Therefore,
\begin{align*}
\frac{1}{\la_n^1}{\rm dist}\,(y_n^1,\pa \Om) \ra \infty \text{ as } n \ra \infty.
\end{align*}
Thus by \eqref{co22} and  Lemma \ref{lem2}, we have $ I_{\infty}^{\prime}(v_1)=0 $ and the sequence
\begin{align*}
u_n^2(x)=u_n^1(x)-\la_n ^{\frac{2-N}{2}}v_1\left(\frac{x-y_n}{\la_n}\right)
\end{align*}
satisfies
\begin{equation*}
\begin{aligned}
& I_{\infty}(u_n^2)\ra c- I_{\infty}(v_0) -I_{\infty}(v_1),\text{ and }  I^{\prime}_{\infty}(u_n^2)\ra 0 \text{ in } (D_0^{1,2}(\Om))^{\prime}.
\end{aligned}
\end{equation*}
By Proposition \ref{prop1}, we have $I_{\infty}(v_1)\geq \beta$. So, iterating the above procedure we can construct sequences $\{v_i\}, \{\la^i_n\}, \{f_n^i\}$ and
after $k$ iterations we obtain
\begin{align*}
I_{\infty}(u_n^{k+1}) < I(u_n) - I(v_0)-\sum_{i=1}^{k}I_{\infty}(v_i) \leq I(u_n) - I(v_0)- k\beta.
\end{align*}
As the later will be negative for large $k$, the induction process terminates after some index $k\geq 0$. Consequently, we get
$k$ sequences $\{y_n^i\}_n \subset \Om $ and $\{\la_n^i\}_n \subset \mathbb{R}_+$, satisfying \eqref{co13}.\QED
\end{proof}
\begin{Definition}
We say that $I$ satisfies the Palais-Smale
condition at $c$ if for any sequence $u_k \in D_0^{1,2}(\Om)$ such that
$I(u_k)\ra c$  and $I^{\prime}(u_k) \ra
0$, then there exists a subsequence that converges strongly in
$ D_0^{1,2}(\Om)$.
\end{Definition}
\begin{Lemma}
The functional  $I$ satisfies Palais-Smale  condition for any  $c \in (\beta, 2\beta)$, where $$\beta= \frac{1}{2}\left(\frac{N-\mu+2}{2N-\mu}\right) S_{H,L}^{\frac{2N-\mu}{N-\mu+2}}.$$
\end{Lemma}
\begin{proof}
For some $c\in (\beta,2 \beta)$, we assume that there exists $\{u_n\}, \in D_0^{1,2}(\Om)$ such that
\begin{align*}
I(u_n)\ra c,  I^{\prime}(u_n)\ra 0 \text{ in } (D_0^{1,2}(\Om))^{\prime}.
\end{align*}
By Lemma \ref{cothm3}, passing to a subsequence (if necessary), there exists a solution $v_0 \in D_0^{1,2}(\Om) $ of  \eqref{co25} and
 $k\in \mathbb{N}\cup \{0\}$, non-trivial solutions $\{v_1,v_2,...,v_k\}$ of \eqref{co26}
 with $v_i \in D^{1,2}(\mathbb{R}^N) $ and $k$ sequences $\{y_n^i\}_n \subset \mathbb{R}^N$ and $\{\la_n^i\}_n \subset \mathbb{R}_+$ satisfying \eqref{co13}.
Now, by equation \eqref{co13} and Proposition \ref{prop1} we have, $k \beta \leq c< 2\beta$. This implies $k\leq 1$.

If $k=0$ compactness holds and we are done.

If $k=1$ then we have two possibilities: either $v_0\not\equiv0$ or $v_0\equiv 0$.   If $v_0 \not\equiv 0$, since $I(v_0)\geq \beta$ and by Lemma 1.3 of \cite{yang},  $\beta$ is never achieved on bounded domain we have $I(v_0)> \beta$
 and this is not possible.  If $v_0\equiv 0$ then by Theorem \ref{cothm2}, $I_{\infty}(v_1)=c$ and $v_1$ is a nonnegative solution of \eqref{co26}.

  Next, by Corollary   \ref{cothm5}, we deduce that $v_1$ is radially symmetric, monotonically decreasing and of the form $v_1(x)= \left(\frac{a}{b+|x-x_0|^2}\right)^{\frac{N-2}{2}}$, for some constants $a,b>0$ and some $x_0 \in \mathbb{R}^N$. Therefore by Lemma \ref{lem5}, we conclude that $S_{H,L}$ is achieved by $v_1$. It follows that $I_{\infty}(v_1)=\beta$, which is a contradiction since $I_{\infty}(v_1)=c> \beta$.
\end{proof}

\section{Proof of Theorem \ref{cothm1}}
 To prove Theorem \ref{cothm1}, we shall first establish some auxiliary results.

 Let $R_1, R_2$ be the radii of the annulus as in Theorem \ref{cothm1}. Without loss of generality, we can assume $x_0=0, R_1=\frac{1}{4R}, R_2= 4R$ where $R>0$  will be chosen sufficiently large. Consider the family of functions
\begin{align*}
 u^{\sigma}_t(x):=S^{\frac{(N-\mu)(2-N)}{4(N-\mu+2)}}C(N,\mu)^{\frac{2-N}{2(N-\mu+2)}} \left(\frac{1-t}{(1-t)^2+|x-t\sigma|^2}\right)^{\frac{N-2}{2}}\in D^{1,2}(\mathbb{R}^N),
\end{align*}
 where  $\sigma\in  \Sigma := \{x\in \mathbb{R}^N:|x|=1\}, t\in [0,1)$. Note that  if $ t \ra 1$ then  $u^{\sigma}_t $ concentrates at $\sigma$. Also, if $ t \ra 0$ then
 \begin{equation*}
u^{\sigma}_t \ra u_0:= S^{\frac{(N-\mu)(2-N)}{4(N-\mu+2)}}C(N,\mu)^{\frac{2-N}{2(N-\mu+2)}}\left(\frac{1}{1+|x|^2}\right)^{\frac{N-2}{2}}.
 \end{equation*}
 Now, define $\upsilon \in C_c^{\infty}(\Om)$ such that $0 \leq \upsilon \leq 1$ on $\Om$ and
$$
\upsilon(x)= \left\{
\begin{array}{lr}
1 \;\quad  \frac{1}{2} <|x|<2 \\
0 \; \quad |x|>4, |x|<\frac{1}{4}.
\end{array}
\right.
$$
Subsequently, we can define
$$
\upsilon_R(x)= \left\{
\begin{array}{lr}
\upsilon(Rx) \;\quad  0 <|x|<\frac{1}{2R} \\
1 \; \qquad \quad \;\frac{1}{2R} \leq |x| \leq R\\
\upsilon(x/R) \; \quad |x|\geq R.
\end{array}
\right.
$$
We now define $$g^{\sigma}_t(x)= u^{\sigma}_t(x)\upsilon _R(x) \in D_0^{1,2}(\Om),\; g_0(x)=u_0(x)\upsilon_R(x).$$
We establish the following auxiliary result.
\begin{Lemma}\label{lem6}
Let $\sigma\in  \Sigma$ and $t\in (0,1]$, then the following holds:
\begin{enumerate}
	\item $\|u^{\sigma}_t\|=\|u_0\|.$
	\item $\|(u^{\sigma}_t)_+\|_{NL}= \|(u_0)_+\|_{NL}.$
	\item $\|u^{\sigma}_t\|^2 =S_{H,L} \|(u^{\sigma}_t)_+\|_{NL}^2$.
	\item $\ds \lim_{R\ra \infty }\ds \sup _{\sigma \in \Sigma, t\in [0,1)}\|g_t^{\sigma}-u_t^{\sigma}\|=0$.
	\item $\ds \lim_{R\ra \infty }\ds \sup _{\sigma \in \Sigma, t\in [0,1)}\|g_t^{\sigma}\|_{NL}^{2.2^*_{\mu}}=\|u_t^{\sigma}\|_{NL}^{2.2^*_{\mu}}$.
	\end{enumerate}

\end{Lemma}
\begin{proof}
By trivial transformations, we can get first two properties $u^{\sigma}_t$  and since $u_t^{\sigma}$ is a minimizer of $S_{H,L}$ therefore, third ones holds.

We have
\begin{equation}\label{co27}
\begin{split}
\int_{\mathbb{R}^N}|\na g_t^{\sigma}-\na u_t^{\sigma} |^2 dx & \leq 2 \int_{\mathbb{R}^N}|u^{\sigma}_t(x) \na \upsilon _R(x)|^2 ~ dx +2 \int_{\mathbb{R}^N}|\na u^{\sigma}_t(x)  \upsilon _R(x)- \na u^{\sigma}_t(x)|^2 ~ dx\\ & \leq  C \left( R^2\int_{B_{\frac{1}{2R}}}|u^{\sigma}_t(x)|^2 ~ dx+ \int_{B_{\frac{1}{2R}}}|\na u^{\sigma}_t(x)|^2 ~ dx \right) \\ & \qquad +
C\left(\frac{1}{R^2} \int_{B_{4R} \setminus B_{2R}}|u^{\sigma}_t(x)|^2 ~ dx+\int_{\mathbb{R}^N \setminus B_{2R}} |\na u^{\sigma}_t(x)|^2 ~ dx \right),
\end{split}
\end{equation}
where $B_\al$ is a ball of radius $\al$ and center $0$.

From the definition of $u^{\sigma}_t$, we have
\begin{align*}
R^2\int_{B_{\frac{1}{2R}}}|u^{\sigma}_t(x)|^2 ~ dx \leq  C R^2\int_{B_{\frac{1}{2R}}} ~ dx \leq \frac{C}{R^{N-2}},
\end{align*}
\begin{align*}
\int_{B_{\frac{1}{2R}}}|\na u^{\sigma}_t(x)|^2 ~ dx \leq C \int_{B_{\frac{1}{2R}}} |x-t\sigma|~ dx \leq C \int_{B_{\frac{1}{2R}}} ~ dx \leq \frac{C}{R^{N}},
\end{align*}
\begin{align*}
\frac{1}{R^2} \int_{B_{4R} \setminus B_{2R}}|u^{\sigma}_t(x)|^2 ~ dx \leq \frac{C}{R^2} \int_{B_{4R} \setminus B_{2R}}\frac{1}{|x|^{2N-4}} ~ dx \leq \frac{C}{R^{N-2}},
\end{align*}
\begin{align*}
\int_{\mathbb{R}^N \setminus B_{2R}} |\na u^{\sigma}_t(x)|^2 ~ dx   \leq C \int_{\mathbb{R}^N  \setminus B_{2R}} \frac{1}{|x|^{2N-2}} ~ dx \leq \frac{C}{R^{N-2}}.
\end{align*}

\noi Therefore, from \eqref{co27} if  $R \ra \infty$ we get $
\ds \sup _{\sigma \in \Sigma, t\in (0,1]}\|g_t^{\sigma}-u_t^{\sigma}\|\ra 0.$

Next,  we shall prove that  $$\ds \lim_{R\ra \infty }\ds \sup _{\sigma \in \Sigma, t\in (0,1]}\|g_t^{\sigma}\|_{NL}^{2.2^*_{\mu}}=\|u_t^{\sigma}\|_{NL}^{2.2^*_{\mu}}.$$
Consider
\begin{align*}
\|g_t^{\sigma}\|_{NL}^{2.2^*_{\mu}}-\|u_t^{\sigma}\|_{NL}^{2.2^*_{\mu}}
& = \int_{\mathbb{R}^N}\int_{\mathbb{R}^N}\frac{(\upsilon_R^{2^*_{\mu}}(x)\upsilon_R^{2^*_{\mu}}(y)-1)|u_t^{\sigma}(x)|^{2^*_{\mu}}|u_t^{\sigma}(y)|^{2^*_{\mu}} } {|x-y|^{\mu}} ~dxdy\\
& \leq C \sum_{i=1}^{5} J_i,
\end{align*}
where \begin{equation*}
\begin{aligned}
& J_1= \int_{B_{2R}\setminus B_{\frac{1}{2R}}}\int_{B_{\frac{1}{2R}}}\frac{|u_t^{\sigma}(x)|^{2^*_{\mu}}|u_t^{\sigma}(y)|^{2^*_{\mu}} } {|x-y|^{\mu}} ~dxdy,\\
& J_2= \int_{B_{2R}\setminus B_{\frac{1}{2R}}}\int_{\mathbb{R}^N\setminus B_{2R}}\frac{|u_t^{\sigma}(x)|^{2^*_{\mu}}|u_t^{\sigma}(y)|^{2^*_{\mu}} } {|x-y|^{\mu}} ~dxdy,\\
& J_3= \int_{ B_{\frac{1}{2R}}}\int_{B_{ \frac{1}{2R}}}\frac{|u_t^{\sigma}(x)|^{2^*_{\mu}}|u_t^{\sigma}(y)|^{2^*_{\mu}} } {|x-y|^{\mu}} ~dxdy,\\
& J_4= \int_{ B_{\frac{1}{2R}}}\int_{\mathbb{R}^N\setminus B_{2R}}\frac{|u_t^{\sigma}(x)|^{2^*_{\mu}}|u_t^{\sigma}(y)|^{2^*_{\mu}} } {|x-y|^{\mu}} ~dxdy,\\
& J_5= \int_{\mathbb{R}^N\setminus B_{2R}}\int_{\mathbb{R}^N\setminus B_{2R}}\frac{|u_t^{\sigma}(x)|^{2^*_{\mu}}|u_t^{\sigma}(y)|^{2^*_{\mu}} } {|x-y|^{\mu}} ~dxdy.
\end{aligned}
\end{equation*}
By the Hardy--Littlewood--Sobolev inequality, we have the following estimates:
\begin{align*}
J_1
& \leq C(N,\mu ) \left(\int_{B_{\frac{1}{2R}}}\frac{(1-t)^Ndx}{((1-t)^2+|x-t\sigma|^2)^N}\right)^{\frac{2N-\mu}{2N}}\left(\int_{B_{2R}\setminus B_{\frac{1}{2R}}}\frac{(1-t)^N dx}{((1-t)^2+|x-t\sigma|^2)^N}\right)^{\frac{2N-\mu}{2N}}\\
& \leq C \left(\int_{B_{\frac{1}{2R}}} (1-t)^{N-2}dx\right)^{\frac{2N-\mu}{2N}} \leq C \left(\frac{1}{2R}\right)^{\frac{2N-\mu}{2}},
\end{align*}
\begin{align*}
J_2
& \leq C(N,\mu) \left(\int_{B_{2R}\setminus B_{\frac{1}{2R}}}\frac{(1-t)^Ndx}{((1-t)^2+|x-t\sigma|^2)^N}\right)^{\frac{2N-\mu}{2N}}\left(\int_{\mathbb{R}^N \setminus B_{2R}}\frac{(1-t)^Ndx}{((1-t)^2+|x-t\sigma|^2)^N}\right)^{\frac{2N-\mu}{2N}}\\
& \leq C \left(\int_{\mathbb{R}^N \setminus B_{2R}}\frac{dx}{|x-t\sigma|^{2N}}\right)^{\frac{2N-\mu}{2N}}\\
& \leq C \left(\int_{|y+t\sigma|\geq 2R}\frac{dy}{|y|^{2N}}\right)^{\frac{2N-\mu}{2N}}\\&  \leq C \left(\int_{|y|\geq 2R-1}\frac{dy}{|y|^{2N}}\right)^{\frac{2N-\mu}{2N}}
\leq C \left(\frac{1}{2R-1}\right)^{\frac{2N-\mu}{2}},
\end{align*}
\begin{align*}
J_3
 \leq C(N,\mu) \left(\int_{B_{\frac{1}{2R}}}\frac{(1-t)^Ndx}{((1-t)^2+|x-t\sigma|^2)^N}\right)^{\frac{2N-\mu}{N}} \leq C \left(\int_{B_{\frac{1}{2R}}}(1-t)^{N-2}dx\right)^{\frac{2N-\mu}{N}} \leq C \left(\frac{1}{2R}\right)^{2N-\mu}.
\end{align*}
Using the same estimates as above we can easily obtain
\begin{align*}
J_4 \leq C \left(\frac{1}{2R}\right)^{\frac{2N-\mu}{2}}\quad \text{ and }\quad  J_5 \leq C \left(\frac{1}{2R-1}\right)^{2N-\mu}.
\end{align*}
This implies that $ \ds \sup _{\sigma \in \Sigma, t\in [0,1)}\left(\|g_t^{\sigma}\|_{NL}^{2.2^*_{\mu}}-\|u_t^{\sigma}\|_{NL}^{2.2^*_{\mu}}\right)\ra 0$  as  $R \ra \infty$ and completes the proof.
\QED
\end{proof}
	In order to proceed further we define the manifold $\mc{M}$ and the functions $G:\mc{M}\rightarrow \mathbb R^N$ as follows: \\
$\ds
 \mathcal{M}= \left\lbrace u \in D_0^{1,2}(\Om) \bigg| \ds \int_{\Om}\int_{\Om}\frac{|u_+(x)|^{2^{*}_{\mu}}|u_+(y)|^{2^{*}_{\mu}}}{|x-y|^{\mu}}~dxdy=1 \right\rbrace,$ and $\displaystyle G(u)= \ds \int_{\Om}x|\na u|^2~dx$.

 We also define $S_{H,L}(u,{\Om}):D^{1,2}_{0}(\Om)\setminus\{0\}  \rightarrow \mathbb{R}$, $S_{H,L}: D^{1,2}(\mathbb R^N)\setminus\{0\}  \rightarrow \mathbb R$ and $\tau:D^{1,2}_{0}(\Om)\rightarrow \mathbb{R}$ as
\begin{align*}
& S_{H,L}(u,\Om)= \frac{\ds \int_{\Om}|\nabla u|^2dx}{\left(\ds \int_{\Om}\int_{\Om}\frac{|u_+(x)|^{2^{*}_{\mu}}|u_+(y)|^{2^{*}_{\mu}}}{|x-y|^{\mu}}~dxdy\right)^{\frac{1}{2^{*}_{\mu}}}}, \;  S_{H,L}(u)= \frac{\ds \int_{\mathbb{R}^N}|\nabla u|^2dx}{\ds \|u_+\|_{NL}^{2} },   \\
&  \text{ and }   \tau(u)=
\left(\ds \int_{\Om}\int_{\Om}\frac{|u_+(x)|^{2^{*}_{\mu}}|u_+(y)|^{2^{*}_{\mu}}}{|x-y|^{\mu}}~dxdy \right)^{\frac{1}{2^*_{\mu}}}.
\end{align*}

\begin{Proposition}\label{prop4}
If $S_{H,L}(. \; ,\Om) \in C^1(D_0^{1,2}(\Om)\setminus \{0\})$ and $S_{H,L}^{\prime}(u,\Om)=0 $ for $u \in D_0^{1,2}(\Om)$ then $I^{\prime}(\la u)=0$ for some $\la >0$.
\end{Proposition}

\begin{proof} Let $w \in D_0^{1,2}(\Om)$ then
\begin{align*}
\ld S_{H,L}& ^{\prime}(u,\Om), \;w \rd
\\ & = \ds \frac{2 \tau(u) \ds \int_{\Om}\na u.\na w~ dx - 2\|u\|^2 \tau(u)^{1-2^*_{\mu}} \ds \int_{\Om}\int_{\Om}\frac{|u_+(x)|^{2^{*}_{\mu}}|u_+(y)|^{2^{*}_{\mu}-2}u_+(y)w(y)}{|x-y|^{\mu}}~dxdy}{\ds \tau(u)^2}.
\end{align*}
As $S_{H,L}^{\prime}(u,\Om)(w)=0$, it implies
\begin{align*}
& \tau(u) \ds \int_{\Om}\na u.\na w~ dx =\|u\|^2 \tau(u)^{1-2^*_{\mu}} \ds \int_{\Om}\int_{\Om}\frac{|u_+(x)|^{2^{*}_{\mu}}|u_+(y)|^{2^{*}_{\mu}-2}u_+(y)w(y)}{|x-y|^{\mu}}~dxdy, \\ & \text{ that is, }  \int_{\Om}\na u.\na w~ dx  = \frac{\|u\|^2 \ds  \int_{\Om}\int_{\Om}\frac{|u_+(x)|^{2^{*}_{\mu}}|u_+(y)|^{2^{*}_{\mu}-2}u_+(y) w(y)}{|x-y|^{\mu}}~dxdy}{\ds \int_{\Om}\int_{\Om}\frac{|u_+(x)|^{2^{*}_{\mu}}|u_+(y)|^{2^{*}_{\mu}}}{|x-y|^{\mu}}~dxdy}.
\end{align*}
Therefore, if we choose $$\la^{2(2^*_{\mu}-1)}=\frac{\ds \|u\|^2 }{\ds \int_{\Om}\int_{\Om}\frac{|u_+(x)|^{2^{*}_{\mu}}|u_+(y)|^{2^{*}_{\mu}}}{|x-y|^{\mu}}~dxdy} $$
then we get  $ I^{\prime}(\la u) =0$. \QED
\end{proof}
\begin{Proposition}\label{prop5}
Let $\{v_n\} \subset \mathcal{M}$ be a Palais-Smale sequence for $S_{H,L}(. \; ,\Om)$ at level $c$. Then $u_n= \la_nv_n, \;\la_n=\left(S_{H,L}(v_n,\Om) \right)^{\frac{N-2}{2(N-\mu+2)}}$ is a Palais-Smale sequence for $I$ at level $\frac{N-\mu+2}{2(2N-\mu)}c^{\frac{2N-\mu}{N-\mu+2}}$.
\end{Proposition}
\begin{proof}
By the calculations of Proposition \ref{prop4} for any $w \in D_0^{1,2}(\Om) $, we have
\begin{align*}
 \frac{1}{2} \ld S_{H,L}^{\prime}(v_n,\Om),\; w\rd & = \ds \int_{\Om}\na v_n.\na w~ dx
 \\&  \; \; -\la_n ^{2(2^*_{\mu}-1)} \ds \int_{\Om}\int_{\Om}\frac{|(v_n)_+(x)|^{2^{*}_{\mu}}|(v_n)_+(y)|^{2^{*}_{\mu}-2}(v_n)_+(y)w(y)}{|x-y|^{\mu}}~dxdy.
\end{align*}
Now by multiplying the above equation by $\la_n$ for any $w \in D_0^{1,2}(\Om) $ we obtain
\begin{align*}
\ld I^{\prime}(u_n),\; w\rd = \ds \int_{\Om}\na u_n.\na w~ dx -  \ds \int_{\Om}\int_{\Om}\frac{|(u_n)_+(x)|^{2^{*}_{\mu}}|(u_n)_+(y)|^{2^{*}_{\mu}-2}(u_n)_+(y)w(y)}{|x-y|^{\mu}}~dxdy.
\end{align*}
\noi Since $v_n\in \mathcal{M}$, therefore $\la^{2(2^*_{\mu}-1)}= \|v_n\|^2= S_{H,L}(v_n,\Om)$ that is, $\la_n=S_{H,L}(v_n,\Om)^{\frac{N-2}{2(N-\mu+2)}}$. From $ S_{H,L}(v_n,\Om)= c+o(1)$ we get  $\la_n$ is bounded. In particular, it follows that $\ld I^{\prime}(\la_nv_n), \; w \rd \ra 0$ as $n \ra \infty$. Also, we have $u_n$ is bounded yields,
\begin{align*}
o(1)= \ld I^{\prime}(u_n), \;u_n\rd= \|u_n\|^2- \ds \int_{\Om}\int_{\Om}\frac{|(u_n)_+(x)|^{2^{*}_{\mu}}|(u_n)_+(y)|^{2^{*}_{\mu}}}{|x-y|^{\mu}}~dxdy.
\end{align*}
All the above facts imply that
\begin{align*}
\lim_{n\ra \infty}I(u_n)= \frac{N-\mu+2}{2(2N-\mu)}\lim_{n\ra \infty} \la_n^{2.2^*_{\mu}}= \frac{N-\mu+2}{2(2N-\mu)}c^{\frac{2N-\mu}{N-\mu+2}}.
\end{align*}

\end{proof}
\begin{Remark} Since we proved $I$ satisfies Palais-Smale  condition in $(\beta,2 \beta)$. Then $S_{H,L}(. \; ,\Om)$ satisfies satisfies Palais-Smale  condition in $\left(S_{H,L}, \; 2^{\frac{N-\mu+2}{2N-\mu}}S_{H,L}\right)$ by using Proposition \ref{prop4}.
\end{Remark}

\begin{Lemma}
If $\ds f^{\sigma}_t(x):= \ds \frac{g^{\sigma}_t(x)}{ \|g_t^{\sigma}\|_{NL}}$ and  $\ds f_0(x):= \ds \frac{g_0(x)}{ \|g_0\|_{NL}}$
then
\begin{align*}
\ds \lim_{R\ra \infty} S_{H,L}(f^{\sigma}_t,\Om)=S_{H,L}(u^{\sigma}_t)= S_{H,L},
\end{align*}
uniformly with respect to $\sigma \in \Sigma $ and $t\in [0,1)$.


%

\end{Lemma}
\begin{proof}
This is a trivial consequence of Lemma \ref{lem6}. \QED
\end{proof}
\noi In particular, if $R> 1 $ sufficiently large then we can achieve that
\begin{align*}
\sup _{\sigma ,t}(f^{\sigma}_t,\Om )< S_1< 2^{\frac{N-\mu+2}{2N-\mu}}S_{H,L} \text{ for some  } S_1 \in \mathbb{R}.
\end{align*}

\noi \textbf{Proof of Theorem \ref{cothm1} completed.} As we have established, $S_{H,L}(. \; ,\Om)$ satisfies Palais-Smale  at level $\alpha$ on $\mathcal{M}$ for $\alpha \in \left(S_{H,L},\; 2^{\frac{N-\mu+2}{2N-\mu}}S_{H,L}\right)$. We will argue by contradiction. If  $S_{H,L}(. \; ,\Om)$ does not admit a critical value in this range. By the deformation lemma (see A.~Bonnet \cite[Theorem 2.5]{bonnet}) for any $\alpha \in \left(S_{H,L},\; 2^{\frac{N-\mu+2}{2N-\mu}}S_{H,L}\right)$ there exist $\de >0 $ and an onto  homeomorphism function  $\psi: \mathcal{M} \ra \mathcal{M}$ such that $\psi(\mathcal{M}_{\alpha+ \de})\subset \mathcal{M}_{\alpha-\de}$ where $\mathcal{M}_{\alpha}=\{u\in \mathcal{M}\; ;\; S_{H,L}(u,\Om)<\alpha\}$. For a given fixed $\e >0$ we can cover the interval $[S_{H,L}+\e,S_1]$ by finitely many such $\de-$ intervals and composing the deformation maps we get an onto  homeomorphism function  $\psi: \mathcal{M}\ra \mathcal{M}$ such that $\psi(\mathcal{M}_{S_1})\subset \mathcal{M}_{S_{H,L}+\e}$.  Also, we can assume  $\psi(u)=u$ for all $u$ whenever $S_{H,L}(u,\Om)\leq S_{H,L}+ \e/2$.

By the concentration-compactness lemma (see \cite{gaoyang}) and Lemma 1.2 of \cite{yang}, for any sequence $\{u_m\} \in \mathcal{M}_{S_{H,L}+\frac{1}{m}}$  there exists a subsequence  and $x^{(0)} \in \overline{\Om}$ such that
\begin{align*}
\left(\int_{\Om}\frac{|(u_m)_+(y)|^{2^*_{\mu}}}{|x-y|^{\mu}}dy\right)|(u_m)_+|^{2^*_{\mu}} dx \rightharpoondown \delta_{x^{(0)}}, \; |\na u_m|^2 dx \rightharpoondown S_{H,L}\delta_{x^{(0)}}
\end{align*}
 weakly in the sense of measure. This implies given any neighbourhood $V$ of $\overline{\Om}$, there exists a $\e>0$ such that $G(\mathcal{M}_{S_{H,L}+\e})\subset V$.\\
Since $\Om$ is a smooth  bounded domain, therefore we can find a  neighbourhood $V$ of $ \overline{\Om}$ such that for any $q\in V$ there exits a unique nearest neighbour $r=\pi(q)\in  \overline{\Om}$ such that the projection $\pi $ is continuous. Let $\e$ be chosen for such a neighbourhood $V$, and let $\psi: \mathcal{M}\ra \mathcal{M}$ be the  corresponding onto homeomorphism. Define the map $D:\Sigma\times [0,1]\ra \overline{\Om} $ given by
\begin{align*}
D(\sigma,t)= \pi\left(G(\psi(f^{\sigma}_t))\right.
\end{align*}
It is easy to see that $D$ is well-defined, continuous and satisfies
\begin{align*}
D(\sigma,0)= \pi\left(G(\psi(f_0))\right)=: y_0 \in \overline{\Om} \text{ and } D(\sigma,1)=\sigma \text{ for all } \sigma \in \Sigma.
\end{align*}
This implies that $D$ is a contraction of $\Sigma$ in $\overline{\Om}$ contradicting the hypothesis of $\Om$.
Hence, our assumption is wrong implies that $S_{H,L}(. \; ,\Om)$ has a critical value that means there exits a $u\in D_0^{1,2}(\Om)$ such that $u$ is a solution to problem $(P)$. Now, using  same arguments and assertions as in \cite[Proposition 3.1]{moroz2}, we have $u \in L^{\infty}(\Om)$. It implies that $|-\De u|\leq C(1+|u|^{2^*-1})$ and from  standard elliptic regularity we have  $u \in  C^2(\overline{\Om})$. Thus, by the maximum principle, $u$ is a positive solution of the problem $(P)$.
 Hence the proof of Theorem \ref{cothm1} is complete.
\QED

\medskip
{\bf Acknowledgements}.
 V.D. R\u adulescu acknowledges the support through the Project
MTM 2017-85449-P of the DGISPI (Spain).

 \end{document}